\documentclass[a4paper,11pt,draft]{article}

\usepackage{a4wide}
\usepackage[utf8]{inputenc}
\usepackage{amsmath,amsfonts,amssymb,amsthm}
\usepackage{mathrsfs}
\usepackage{color}
\usepackage{stmaryrd}

\usepackage[hidelinks, draft=false]{hyperref}

\theoremstyle{plain}
\newtheorem{thm}{Theorem}[section]
\newtheorem{lem}[thm]{Lemma}
\newtheorem{cor}[thm]{Corollary}
\newtheorem{prop}[thm]{Propostion}

\theoremstyle{definition}
\newtheorem{defn}[thm]{Definition}

\theoremstyle{remark}
\newtheorem{rmk}[thm]{Remark}
\newtheorem{ex}[thm]{Example}

\newcommand{\dd}{\mathrm{d}}
\newcommand{\R}{\mathbb{R}}
\newcommand{\F}{\mathcal{F}}
\newcommand{\N}{\mathbb{N}}
\renewcommand{\L}{\mathcal{L}}

\newcommand{\1}{\mathbf{1}}

\newcommand{\C}{\mathcal{C}}
\newcommand{\A}{\mathcal{A}}
\newcommand{\D}{\mathcal{D}}
\newcommand{\ito}{I^{\mathrm{It\hat{o}}}}

\newcommand{\dD}{\mathrm{D}}

\title{A Fourier approach to pathwise stochastic integration}

\author{
  Massimiliano Gubinelli\thanks{M.G. is supported by a Junior fellowship of the Institut Universitaire de France (IUF) and  by the ANR Project
ECRU (ANR-09-BLAN-0114-01)}
 \\ CEREMADE \& CNRS UMR 7534 \\ Universit{\'e} Paris-Dauphine\\
 and Institut Universitaire de France \\
  \texttt{gubinelli@ceremade.dauphine.fr}
  \and
  Peter Imkeller \\ Institut f\"ur Mathematik \\ Humboldt-Universit\"at zu Berlin\\
  \texttt{imkeller@math.hu-berlin.de}
  \and
  Nicolas Perkowski\thanks{N.P. is supported by the Fondation Sciences Math\'ematiques de Paris (FSMP) and by a public grant overseen by the French National Research Agency (ANR) as part of the ``Investissements d'Avenir'' program (reference: ANR-10-LABX-0098), and acknowledges generous support from Humboldt-Universit\"at zu Berlin, where a major part of this work was completed.}
 \\  CEREMADE \& CNRS UMR 7534  \\ Universit{\'e} Paris-Dauphine\\
  \texttt{perkowski@ceremade.dauphine.fr}
}

\begin{document}
\maketitle

\begin{abstract}
We develop a Fourier approach to rough path integration, based on the series decomposition of continuous functions in terms of Schauder functions. Our approach is rather elementary, the main ingredient being a simple commutator estimate, and it leads to recursive algorithms for the calculation of pathwise stochastic integrals, both of It\^o and of Stratonovich type. We apply it to solve stochastic differential equations in a pathwise manner. %In the setting of It\^{o} integration, we show that under suitable conditions, the It\^{o} rough path integral can be obtained as limit of nonanticipating Riemann sums involving only the integrator and not its iterated integrals.
\end{abstract}

\tableofcontents

\section{Introduction}

The theory of rough paths~\cite{Lyons1998} has recently been extended to a multiparameter setting~\cite{Hairer2014Regularity,Gubinelli2012}. While \cite{Hairer2014Regularity} has a much wider range of applicability, both approaches allow to solve many interesting SPDEs that were well out of reach with previously existing methods; for example the continuous parabolic Anderson in dimension two~\cite{Hairer2014Regularity,Gubinelli2012}, the three-dimensional stochastic quantization equation~\cite{Hairer2014Regularity,Catellier2013}, the KPZ equation~\cite{Hairer2013KPZ,Gubinelli2014}, or the three-dimensional stochastic Navier Stokes equation~\cite{Zhu2014,Zhu2014Discretization}. Our methods developed in~\cite{Gubinelli2012} are based on harmonic analysis, on Littlewood-Paley decompositions of tempered distributions, and on a simple commutator lemma. This requires a non-negligible knowledge of Littlewood-Paley theory and Besov spaces, while at the same time the application to classical rough path SDEs is not quite straightforward. That is why here we develop the approach of~\cite{Gubinelli2012} in the slightly different language of Haar / Schauder functions, which allows us to communicate our basic ideas while requiring only very basic knowledge in analysis. Moreover, in the Haar Schauder formulation the application to SDEs poses no additional technical challenges.

It is a classical result of Ciesielski~\cite{Ciesielski1960} that $C^\alpha := C^\alpha([0,1],\R^d)$, the space of $\alpha$--H\"older continuous functions on $[0,1]$ with values in $\R^d$, is isomorphic to $\ell^\infty(\R^d)$, the space of bounded sequences with values in $\R^d$. The isomorphism gives a Fourier decomposition of a H\"older-continuous function $f$ as
\begin{align*}
   f = \sum_{p,m} \langle H_{pm}, \dd f \rangle G_{pm},
\end{align*}
where $(H_{pm})$ are the Haar functions and $(G_{pm})$ are the Schauder functions. Ciesielski proved that a continuous function $f$ is in $C^{\alpha}([0,1],\R^d)$ if and only if the coefficients $(\langle H_{pm}, \dd f\rangle)_{p,m}$ decay rapidly enough. % satisfy $\sup_{p,m} 2^{p(\alpha - 1/2)} |\langle H_{pm}, \dd f\rangle| < \infty$.
Following Ciesielski's work, similar isomorphisms have been developed for many Fourier and wavelet bases, showing that the regularity of a function is encoded in the decay of its coefficients in these bases; see for example Triebel~\cite{Triebel2006}.

But until this day, the isomorphism based on Schauder functions plays a special role in stochastic analysis, because the coefficients in the Schauder basis have the pleasant property that they are just rescaled second order increments of $f$. So if $f$ is a stochastic process with known distribution, then also the distribution of its coefficients in the Schauder basis is known explicitly. %This makes the Schauder functions a very useful tool in stochastic analysis.
A simple application is the L\'evy-Ciesielski construction of Brownian motion. An incomplete list with further applications will be given below.

Another convenient property of Schauder functions is that they are piecewise linear, and therefore their iterated integrals $\int_0^\cdot G_{pm}(s) \dd G_{qn}(s)$, can be easily calculated. This makes them an ideal tool for our purpose of studying integrals. Indeed, given two continuous functions $f$ and $g$ on $[0,1]$ with values in $\L(\R^d, \R^n)$, the space of linear maps from $\R^d$ to $\R^n$, and $\R^d$ respectively, we can formally define
\begin{align*}
   \int_0^t f(s) \dd g(s) := \sum_{p,m}\sum_{q,n} \langle H_{pm}, \dd f \rangle \langle H_{qn}, \dd g \rangle \int_0^t G_{pm} (s) \dd G_{qn}(s).
\end{align*}
In this paper we study, under which conditions this formal definition can be made rigorous. We start by observing that the integral introduces a bounded operator from $C^\alpha \times C^\beta$ to $C^\beta$ if and only if $\alpha+\beta > 1$. Obviously, here we simply recover Young's integral~\cite{Young1936}. In our study of this integral, we identify different components:
\begin{align*}
   \int_0^t f(s) \dd g(s) = S(f,g)(t) + \pi_<(f,g)(t) + L(f,g)(t),
\end{align*}
where $S$ is the \emph{symmetric part}, $\pi_<$ the \emph{paraproduct}, and $L(f,g)$ the \emph{L\'evy area}. The operators $S$ and $\pi_<$ are defined for $f \in C^\alpha$ and $g \in C^\beta$ for arbitrary $\alpha,\beta>0$, and it is only the L\'evy area which requires $\alpha + \beta > 1$. Considering the regularity of the three operators, we have $S(f,g) \in C^{\alpha + \beta}$, $\pi_<(f,g) \in C^\beta$, and $L(f,g) \in C^{\alpha+\beta}$ whenever the latter is defined. Therefore, in the Young regime $\int_0^\cdot f(s) \dd g(s) - \pi_<(f,g) \in C^{\alpha + \beta}$. We will also see that for sufficiently smooth functions $F$ we have $F(f) \in C^{\alpha}$ but $F(f) - \pi_<(\dD F(f), f) \in C^{2\alpha}$. So both $\int_0^\cdot f(s) \dd g(s)$ and $F(f)$ are given by a paraproduct plus a smoother remainder. This leads us to call a function $f \in C^\alpha$ \emph{paracontrolled} by $g$ if there exists a function $f^g \in C^\beta$ such that $f - \pi_<(f^g,g) \in C^{\alpha+\beta}$. Our aim is then to construct the L\'evy area $L(f,g)$ for $\alpha < 1/2$ and $f$ paracontrolled by $g$. If $\beta > 1/3$, then the term $L(f - \pi_<(f^g,g),g)$ is well defined, and it suffices to make sense of the term $L(\pi_<(f^g,g),g)$. This is achieved with the following commutator estimate:
\begin{align*}
   \left\lVert L(\pi_<(f^g,g),g) - \int_0^\cdot f^g(s) \dd L(g,g)(s)\right\rVert_{3\beta} \le \lVert f^g \rVert_\beta \lVert g \rVert_\beta \lVert g \rVert_\beta.
\end{align*}
Therefore, the integral $\int_0^\cdot f(s)\dd g(s)$ can be constructed for all $f$ that are paracontrolled by $g$, provided that $L(g,g)$ can be constructed. In other words, we have found an alternative formulation of Lyons'~\cite{Lyons1998} rough path integral, at least for H\"older continuous functions of H\"older exponent larger than 1/3.

Since we approximate $f$ and $g$ by functions of bounded variation, our integral is of Stratonovich type, that is it satisfies the usual integration by parts rule. We also consider a non-anticipating It\^{o} type integral, that can essentially be reduced to the Stratonovich case with the help of the quadratic variation.

The last remaining problem is then to construct the L\'evy area $L(g,g)$ for suitable stochastic processes $g$. We construct it for certain hypercontractive processes. For continuous martingales that possess sufficiently many moments we give a construction of the It\^{o} iterated integrals that allows us to use them as integrators for our pathwise It\^{o} integral.

Below we give some pointers to the literature, and we introduce some basic notations which we will use throughout. In Section~\ref{s:preliminaries ciesielski} we recall some details on Ciesielski's isomorphism, and we give a short overview on rough paths and Young integration. In Section~\ref{s:paradifferential calculus} we develop a paradifferential calculus in terms of Schauder functions, and we examine the different components of Young's integral. In Section~\ref{s:schauder rough path integral} we construct the rough path integral based on Schauder functions. Section~\ref{s:pathwise ito} develops the pathwise It\^o integral. In Section~\ref{s:construction of levy area} we construct the L\'evy area for suitable stochastic processes. And in Section~\ref{s:sde} we apply our integral to solve both It\^o type and Stratonovich type SDEs in a pathwise way.

\paragraph{Relevant literature}

Starting with the L\'evy-Ciesielski construction of Brownian motion, Schauder functions have been a very popular tool in stochastic analysis. They can be used to prove in a comparatively easy way that stochastic processes belong to Besov spaces; see for example Ciesielski, Kerkyacharian, and Roynette~\cite{Ciesielski1993}, Roynette~\cite{Roynette1993}, and Rosenbaum~\cite{Rosenbaum2009}. Baldi and Roynette~\cite{Baldi1992} have used Schauder functions to extend the large deviation principle for Brownian motion from the uniform to the H\"older topology; see also Ben Arous and Ledoux~\cite{BenArous1994} for the extension to diffusions, Eddahbi, N'zi, and Ouknine~\cite{Eddahbi1999} for the large deviation principle for diffusions in Besov spaces, and Andresen, Imkeller, and Perkowski~\cite{Andresen2013} for the large deviation principle for a Hilbert space valued Wiener process in H\"older topology. Ben Arous, Gr\u{a}dinaru, and Ledoux~\cite{BenArous1994a} use Schauder functions to extend the Stroock-Varadhan support theorem for diffusions from the uniform to the H\"older topology. Lyons and Zeitouni~\cite{Lyons1999} use Schauder functions to prove exponential moment bounds for Stratonovich iterated integrals of a Brownian motion conditioned to stay in a small ball. Gantert~\cite{Gantert1994} uses Schauder functions to associate to every sample path of the Brownian bridge a sequence of probability measures on path space, and continues to show that for almost all sample paths these measures converge to the distribution of the Brownian bridge. This shows that the law of the Brownian bridge can be reconstructed from a single ``typical sample path''.

Concerning integrals based on Schauder functions, there are three important references: Roynette~\cite{Roynette1993} constructs a version of Young's integral on Besov spaces and shows that in the one dimensional case the Stratonovich integral $\int_0^\cdot F(W_s) \dd W_s$, where $W$ is a Brownian motion, and $F \in C^2$, can be defined in a deterministic manner with the help of Schauder functions. Roynette also constructs more general Stratonovich integrals with the help of Schauder functions, but in that case only almost sure convergence is established, where the null set depends on the integrand, and the integral is not a deterministic operator. Ciesielski, Kerkyacharian, and Roynette~\cite{Ciesielski1993} slightly extend the Young integral of~\cite{Roynette1993}, and simplify the proof by developing the integrand in the Haar basis and not in the Schauder basis. They also construct pathwise solutions to SDEs driven by fractional Brownian motions with Hurst index $H>1/2$. Kamont~\cite{Kamont1994} extends the approach of~\cite{Ciesielski1993} to define a multiparameter Young integral for functions in anisotropic Besov spaces. Ogawa~\cite{Ogawa1984, Ogawa1985} investigates an integral for anticipating integrands he calls \emph{noncausal} starting from a Parseval type relation in which integrand and Brownian motion as integrator are both developed by a given complete orthonormal system in the space of square integrable functions on the underlying time interval. This concept is shown to be strongly related to  Stratonovich type integrals (see Ogawa~\cite{Ogawa1985}, Nualart, Zakai~\cite{NualartZakai1989}), and used to develop a stochastic calculus on a Brownian basis with \emph{noncausal} SDE (Ogawa~\cite{Ogawa2007}). 

Rough paths have been introduced by Lyons~\cite{Lyons1998}, see also~\cite{Lyons1995,Lyons1996,Lyons1997} for previous results. Lyons observed that solution flows to SDEs (or more generally ordinary differential equations (ODEs) driven by rough signals) can be defined in a pathwise, continuous way if paths are equipped with sufficiently many iterated integrals. More precisely, if a path has finite $p$--variation for some $p \ge 1$, then one needs to associate $\lfloor
p\rfloor$ iterated integrals to it to obtain an object which can be taken as the driving signal in an ODE, such that the solution to the ODE depends continuously on that signal. Gubinelli~\cite{Gubinelli2004, Gubinelli2010} simplified the theory of rough paths by introducing the concept of controlled paths, on which we will strongly rely in what follows. Roughly speaking, a path $f$ is controlled by the reference path $g$ if the small scale fluctuations of $f$ ``look like those of $g$''. Good monographs on rough paths are~\cite{Lyons2002, Lyons2007, Friz2010, Friz2013}.

\paragraph{Notation and conventions.}%\label{sec:notation}

Throughout the paper, we use the notation $a \lesssim b$ if there exists a constant $c>0$, independent of the variables under consideration, such that $a \leqslant c \cdot b$, and we write $a \simeq b$ if $a \lesssim b$ and $b \lesssim a$. If we want to emphasize the dependence of $c$ on the variable $x$, then we write $a(x) \lesssim_{x} b(x)$.

For a multi-index $\mu = ( \mu_{1} , \ldots , \mu_{d} ) \in \mathbb{N}^{d}$ we write $| \mu | = \mu_{1} + \ldots + \mu_{d}$ and $\partial^{\mu} = \partial^{| \mu |} / \partial_{x_{1}}^{\mu_{1}} \cdots \partial_{x_{d}}^{\mu_{d}}$. %{\color{red} For $x= ( x_{1} , \ldots ,x_{d} ) \in \mathbb{R}^{d}$ we write $x^{\mu} =x_{1}^{\mu_{1}} \cdots x_{d}^{\mu_{d}}$.}
$\dD F$ or $F'$ denote the total derivative of $F$. For $k \in \mathbb{N}$ we denote by $\dD^{k} F$ the $k$-th order derivative of $F$. We also write $\partial_{x}$ for the partial derivative in direction $x$.

\section{Preliminaries}\label{s:preliminaries ciesielski}

\subsection{Ciesielski's isomorphism}\label{s:ciesielski}

Let us briefly recall Ciesielski's isomorphism between $C^\alpha([0,1],\R^d)$ and $\ell^\infty(\R^d)$. The \emph{Haar functions} $(H_{pm}, p \in \N, 1 \le m \le 2^p)$ are defined as
\begin{align*}
   H_{pm}(t) :=  \begin{cases}
                                     \sqrt{2^p}, & t \in  \left[ \frac{m-1}{2^{p}}, \frac{2m-1}{2^{p+1}}\right),\\
                                     -\sqrt{2^p}, & t \in \left[ \frac{2m-1}{2^{p+1}}, \frac{m}{2^{p}}\right), \\
                                     0, & \text{otherwise.}
                                  \end{cases}
\end{align*}
When completed by $H_{00} \equiv 1$, the Haar functions are an orthonormal basis of $L^2([0,1],\dd t)$. For convencience of notation, we also define $H_{p0}\equiv 0$ for $p \ge 1$. The primitives of the Haar functions are called \emph{Schauder functions} and they are given by $G_{pm} (t) := \int_0^t H_{pm} (s) \dd s$ for $t\in[0,1]$, $p\in \N$, $0 \le m \le 2^p$. More explicitly, $G_{00}(t) = t$ and for $p\in \N$, $1 \le m \le 2^p$
\begin{align*}
   G_{pm} (t) = \begin{cases}
                                     2^{p/2}\left(t - \frac{m-1}{2^{p}}\right), & t \in  \left[ \frac{m-1}{2^{p}}, \frac{2m-1}{2^{p+1}}\right),\\
                                     - 2^{p/2}\left(t - \frac{m}{2^{p}} \right), & t \in \left[ \frac{2m-1}{2^{p+1}}, \frac{m}{2^{p}}\right),\\
                                     0, & \text{otherwise}.
                                 \end{cases}
\end{align*}
Since every $G_{pm}$ satisfies $G_{pm}(0) = 0$, we are only able to expand functions $f$ with $f(0)=0$ in terms of this family $(G_{pm})$. Therefore, we complete $(G_{pm})$ once more, by defining $G_{-10}(t) := 1$ for all $t \in [0,1]$. To abbreviate notation, we define the times $t^i_{pm}$, $i = 0,1,2$, as
\begin{align*}
   t_{pm}^0 := \frac{m-1}{2^p}, \quad t_{pm}^1 := \frac{2m-1}{2^{p+1}}, \quad t_{pm}^2 := \frac{m}{2^p},
\end{align*}
for $p \in \N$ and $1 \le m \le 2^p$. Further, we set $t^0_{-10} := 0$, $t^1_{-10}:= 0$, $t^2_{-10}:=1$, and $t^0_{00}:=0$, $t^1_{00}:=1$, $t^2_{00}:=1$, as well as $t^i_{p0} := 0$ for $p \ge 1$ and $i = 0,1,2$. The definition of $t^i_{-10}$ and $t^i_{00}$ for $i\neq 1$ is rather arbitrary, but the definition for $i = 1$ simplifies for example the statement of Lemma~\ref{l:schauder functions give linear interpolation} below.

For $f \in C([0,1],\R^d)$, $p\in \N$, and $1 \le m \le 2^p$, we write
\begin{align*}
   \langle H_{pm}, \dd f \rangle :=\,& 2^{\frac{p}{2}}\left[ \left(f\left(t^1_{pm}\right) - f\left(t^0_{pm}\right)\right) - \left( f\left(t^2_{pm}\right) - f\left(t^1_{pm}\right)\right)\right] \\
   =\, & 2^{\frac{p}{2}}\left[ 2 f\left(t^1_{pm}\right) - f\left(t^0_{pm}\right) - f\left(t^2_{pm}\right)\right]
\end{align*}
and $\langle H_{00}, \dd f \rangle := f(1) - f(0)$ as well as $\langle H_{-10}, \dd f \rangle := f(0)$. Note that we only defined $G_{-10}$ and not $H_{-10}$.%, so that the definition of $\langle H_{-10}, \dd f\rangle$ is to be understood as convention.

%The following statement is easily shown by induction.

\begin{lem}\label{l:schauder functions give linear interpolation}
   For $f\colon[0,1]\rightarrow \R^d$, the function
   \[
      f_k := \langle H_{-10}, \dd f\rangle G_{-10} + \langle H_{00}, \dd f \rangle G_{00} + \sum_{p=0}^k \sum_{m=1}^{2^p} \langle H_{pm}, \dd f \rangle G_{pm} = \sum_{p=-1}^k \sum_{m=0}^{2^p} \langle H_{pm}, \dd f \rangle G_{pm}
   \]
   is the linear interpolation of $f$ between the points $t^1_{-10}, t^1_{00}, t^1_{pm}$, $0 \le p \le k, 1 \le m \le 2^p$. If $f$ is continuous, then $(f_k)$ converges uniformly to $f$ as $k \rightarrow \infty$.
\end{lem}

%\begin{proof}
%   The statement follows easily by induction.
%\end{proof}

Ciesielski~\cite{Ciesielski1960} observed that if $f$ is H\"older-continuous, then the series $(f_k)$ converges absolutely and the speed of convergence can be estimated in terms of the H\"older norm of $f$. The norm $\lVert \cdot \rVert_{C^\alpha}$ is defined as
\[
   %\lVert f \rVert_\infty := \sup_{t \in [0,1]} |f(t)| \qquad \text{and}\qquad \lVert f \rVert_{C^\alpha} :=
   \lVert f \rVert_{C^\alpha} := \lVert f \rVert_\infty + \sup_{0\le s < t \le 1} \frac{|f_{s,t}|}{|t-s|^\alpha},
\]
where we introduced the notation
   \[
      f_{s,t} := f(t) - f(s).
   \]

\begin{lem}[\cite{Ciesielski1960}]\label{l:ciesielski}
   Let $\alpha \in (0,1)$. A continuous function $f: [0,1] \rightarrow \R^d$ is in $C^\alpha$ if and only if $\sup_{p,m} 2^{p(\alpha - 1/2)} |\langle H_{pm}, \dd f\rangle| < \infty$. In this case
   \begin{gather}\label{e:ciesielski isomorphism}
      \sup_{p,m} 2^{p(\alpha - 1/2)} |\langle H_{pm}, \dd f\rangle| \simeq \lVert f \rVert_\alpha \text{ and} \\ \nonumber
      \lVert f - f_{N-1} \rVert_\infty = \Big\lVert \sum_{p = N}^\infty \sum_{m=0}^{2^p} |\langle H_{pm}, \dd f\rangle|  G_{pm} \Big\rVert_\infty \lesssim \lVert f \rVert_\alpha 2^{-\alpha N}.
   \end{gather}
\end{lem}

Before we continue, let us slightly change notation. We want to get rid of the factor $2^{-p/2}$ in \eqref{e:ciesielski isomorphism}, and therefore we define for $p \in \N$ and $0 \le m \le 2^p$ the rescaled functions
\begin{align*}
   \chi_{pm} := 2^{\frac{p}{2}} H_{pm} \qquad \text{and} \qquad \varphi_{pm} := 2^{\frac{p}{2}} G_{pm},
\end{align*}
as well as $\varphi_{-10} := G_{-10} \equiv 1$. Then we have for $p \in \N$ and $1 \le m \le 2^p$
\begin{align*}
   \lVert\varphi_{pm}(t)\rVert_\infty = \varphi_{pm}(t^1_{pm}) = 2^{\frac{p}{2}} \int_{t^0_{pm}}^{t^1_{pm}} 2^{\frac{p}{2}} \dd s = 2^p \left( \frac{2m-1}{2^{p+1}} - \frac{2m - 2}{2^{p+1}}\right) = \frac{1}{2},
\end{align*}
so that $\lVert \varphi_{pm}\rVert_\infty \le 1$ for all $p,m$. The expansion of $f$ in terms of $(\varphi_{pm})$ is given by $f_k = \sum_{p=0}^k \sum_{m=0}^{2^p} f_{pm} \varphi_{pm}$, where $f_{-10} := f(1)$, and $f_{00} := f(1)-f(0)$ and for $p \in \N$ and $m \ge 1$
\begin{align*}
   f_{pm} := 2^{-p} \langle \chi_{pm}, \dd f \rangle =  2 f\left(t^1_{pm}\right) - f\left(t^0_{pm}\right) - f\left(t^2_{pm}\right) = f_{t^0_{pm}, t^1_{pm}} - f_{t^1_{pm}, t^2_{pm}}.
\end{align*}
We write $\langle \chi_{pm}, \dd f\rangle := 2^p f_{pm}$ for all values of $(p,m)$, despite not having defined $\chi_{-10}$.

%We will usually measure the regularity of functions by the size of their coefficients in the Schauder expansion:

\begin{defn}
   For $\alpha > 0$ and $f \colon [0,1] \to \R^d$ the norm $\lVert \cdot \rVert_{\alpha}$ is defined as
   \[
      \lVert f \rVert_\alpha := \sup_{pm} 2^{p\alpha} |f_{pm}|,
   \]
   and we write
   \begin{align*}
      \C^\alpha := \C^\alpha(\R^d) := \left\{f \in C( [0,1], \R^d): \lVert f \rVert_\alpha < \infty\right\}.
   \end{align*}
\end{defn}

The space $\C^\alpha$ is isomorphic to $\ell^\infty(\R^d)$, in particular it is a Banach space. For $\alpha \in (0,1)$, Ciesielski's isomorphism (Lemma~\ref{l:ciesielski}) states that $\C^\alpha = C^\alpha([0,1],\R^d)$. Moreover, it can be shown that $\C^1$ is the Zygmund space of continuous functions $f$ satisfying $|2f(x) - f(x+h) - f(x-h)| \lesssim h$. But for $\alpha > 1$, there is no reasonable identification of $\C^{\alpha}$ with a classical function space. For example if $\alpha \in (1,2)$, the space $C^{\alpha}([0,1], \R^d)$ consists of all continuously differentiable functions $f$ with $(\alpha-1)$--H\"older continuous derivative $\dD f$. Since the tent shaped functions $\varphi_{pm}$ are not continuously differentiable, even an $f$ with a finite Schauder expansion is generally not in $C^{\alpha}$.%, despite being in $\C^\beta$ for all $\beta > 0$.

The a priori requirement of $f$ being continuous can be relaxed, but not much. Since the coefficients $(f_{pm})$ evaluate the function $f$ only in countably many points, a general $f$ will not be uniquely determined by its expansion. But for example it would suffice to assume that $f$ is c\`adl\`ag.

% To obtain continuity of $f$ from its coefficients $(f_{pm})$ is only possible if $f$ is uniquely determined by the values $(f(t^i_{pm}))_{i,p,m}$. This is the case if $f$ is right- or left-continuous, but in general it is false, because we may always choose a point $t_0$ that is not dyadic and define $\widetilde{f}(t) := f(t)$ for all $t \neq t_0$, and $\widetilde{f}(t_0) := f(t_0) + 1$. Since the set $(t^i_{pm})_{i,p,m}$ is countable, it is not even true that the coefficients of $f$ determine the function Lebesgue-almost everywhere.

\paragraph{Littlewood-Paley notation.}

We will employ notation inspired from Littlewood-Paley theory. For $p \ge -1$ and $f \in C([0,1])$ we define
\begin{align*}
   \Delta_p f := \sum_{m=0}^{2^p} f_{pm} \varphi_{pm} \qquad \text{and} \qquad S_p f := \sum_{q \le p} \Delta_q f.
\end{align*}
We will occasionally refer to $(\Delta_p f)$ as the Schauder blocks of $f$. Note that
\[
   \C^\alpha = \{f \in C([0,1],\R^d): \lVert (2^{p\alpha} \lVert \Delta_p f \rVert_\infty)_p \rVert_{\ell^\infty} < \infty\}.
\]

\subsection{Young integration and rough paths} \label{s:rough paths}

Here we present the main concepts of Young integration and of rough path theory. The results presented in this section will not be applied in the remainder of this chapter, but we feel that it could be useful for the reader to be familiar with the basic concepts of rough paths, since it is the main inspiration for the constructions developed below.

Young's integral~\cite{Young1936} allows to define $\int f \dd g$ for $f \in C^\alpha$, $g \in C^\beta$, and $\alpha + \beta > 1$. More precisely, let $f \in C^\alpha$ and $g \in C^\beta$ be given, let $t \in [0,1]$, and let $\pi = \{t_0, \dots, t_N\}$ be a partition of $[0,t]$, i.e. $0=t_0 < t_1 < \dots < t_N=t$. Then it can be shown that the Riemann sums
\begin{align*}
   \sum_{t_k \in \pi} f(t_k) (g(t_{k+1})-g(t_k)) := \sum_{k=0}^{N-1} f(t_k) (g(t_{k+1})-g(t_k))
\end{align*}
converge as the mesh size $\max_{k=0,\dots, N-1} |t_{k+1}-t_k|$ tends to zero, and that the limit does not depend on the approximating sequence of partitions. We denote the limit by $\int_0^t f(s) \dd g(s)$, and we define $\int_s^t f(r) \dd g(r) := \int_0^t f(r) \dd g(r) - \int_0^s f(r) \dd g(r)$. The function $t \mapsto \int_0^t f(s) \dd g(s)$ is uniquely characterized by the fact that
\begin{align*}
   \left| \int_s^t f(r) \dd g(r) - f(s) (g(t)-g(s)) \right| \lesssim |t-s|^{\alpha + \beta} \lVert f \rVert_\alpha \lVert g \rVert_\beta
\end{align*}
for all $s,t \in [0,1]$. The condition $\alpha + \beta > 1$ is sharp, in the sense that there exist $f, g \in C^{1/2}$, and a sequence of partitions $(\pi_n)_{n \in \N}$ with mesh size going to zero, for which the Riemann sums $\sum_{t_k \in \pi_n} f(t_k) (g(t_{k+1})-g(t_k))$ do not converge as $n$ tends to $\infty$.

The condition $\alpha + \beta > 1$ excludes one of the most important examples: we would like to take $g$ as a sample path of Brownian motion, and $f = F(g)$. Lyons' theory of rough paths~\cite{Lyons1998} overcomes this restriction by stipulating the ``existence'' of basic integrals and by defining a large class of related integrals as their functionals. Here we present the approach of Gubinelli~\cite{Gubinelli2004}.

Let $\alpha \in (1/3,1)$ and assume that we are given two functions $v,w \in C^\alpha$, as well as an associated ``Riemann integral'' $I^{v,w}_{s,t} = \int_s^t v(r) \dd w(r)$ that satisfies the estimate
\begin{align}\label{e:area estimate}
   |\Phi^{v,w}_{s,t}|:=|I^{v,w}_{s,t} - v(s) w_{s,t}| \lesssim |t-s|^{2\alpha}.
\end{align}
The remainder $\Phi^{v,w}$ is often (incorrectly) called the \emph{area} of $v$ and $w$. This name has its origin in the fact that its antisymmetric part $1/2(\Phi^{v,w}_{s,t} - \Phi^{w,v}_{s,t})$ corresponds to the algebraic area spanned by the curve $((v(r), w(r)): r \in [s,t])$ in the plane $\R^2$.

If $\alpha \le 1/2$, then the integral $I^{v,w}$ cannot be constructed using Young's theory of integration, and also $I^{v,w}$ is not uniquely characterized by \eqref{e:area estimate}. But let us assume nonetheless that we are given such an integral $I^{v,w}$ satisfying \eqref{e:area estimate}. A function $f \in C^\alpha$ is \emph{controlled} by $v \in C^\alpha$ if there exists $f^v \in C^\alpha$, such that for all $s,t \in [0,1]$
\begin{align}\label{e:controlled}
   |f_{s,t} - f^v_s v_{s,t}| \lesssim |t-s|^{2\alpha}.
\end{align}

\begin{prop}[\cite{Gubinelli2004}, Theorem 1]\label{p:Gubinelli rough paths}
   Let $\alpha > 1/3$, let $v,w \in C^\alpha$, and let $I^{v,w}$ satisfy \eqref{e:area estimate}. Let $f$ and $g$ be controlled by $v$ and $w$ respectively, with derivatives $f^v$ and $g^w$. Then there exists a unique function $I(f,g) = \int_0^\cdot f(s) \dd g(s)$ that satisfies for all $s,t \in [0,1]$
   \begin{align*}
      |I(f,g)_{s,t} - f(s) g_{s,t} - f^v(s) g^w(s) \Phi^{v,w}_{s,t}| \lesssim |t-s|^{3\alpha}.
   \end{align*}
   If $(\pi_n)$ is a sequence of partitions of $[0,t]$, with mesh size going to zero, then
   \begin{align*}
      I(f,g)(t) = \lim_{n \rightarrow \infty} \sum_{t_k \in \pi_n} \left( f(t_k) g_{t_k, t_{k+1}} + f^v_{t_k} g^w_{t_k} \Phi^{v,w}_{t_k, t_{k+1}}\right).
   \end{align*}
\end{prop}

The integral $I(f,g)$ coincides with the Riemann-Stieltjes integral and with the Young integral, whenever these are defined. Moreover, the integral map is self-consistent, in the sense that if we consider $v$ and $w$ as paracontrolled by themselves, with derivatives $v^v = w^w \equiv 1$, then $I(v,w) = I^{v,w}$.

The only remaining problem is the construction of the integral $I^{v,w}$. This is usually achieved with probabilistic arguments. If $v$ and $w$ are Brownian motions, then we can for example use It\^{o} or Stratonovich integration to define $I^{v,w}$. Already in this simple example we see that the integral $I^{v,w}$ is not unique if $v$ and $w$ are outside of the Young regime.

It is possible to go beyond $\alpha > 1/3$ by stipulating the existence of higher order iterated integrals. For details see~\cite{Gubinelli2010} or any book on rough paths, such as~\cite{Lyons2002,Lyons2007,Friz2010,Friz2013}.

%Note that the rough path integral is similar in spirit to F\"ollmer's pathwise It\^o calculus, see Chapter~\ref{ch:vovk}, that stipulates the existence of the quadratic variation and uses this to give a pathwise construction of stochastic integrals.

\section{Paradifferential calculus and Young integration}\label{s:paradifferential calculus}

In this section we develop the basic tools that will be required for our rough path integral in terms of Schauder functions, and we study Young's integral and its different components.

\subsection{Paradifferential calculus with Schauder functions}

Here we introduce a ``paradifferential calculus'' in terms of Schauder functions. Paradifferential calculus is usually formulated in terms of Littlewood-Paley blocks and was initiated by Bony~\cite{Bony1981}. For a gentle introduction see~\cite{Bahouri2011}.

We will need to study the regularity of $\sum_{p,m} u_{pm} \varphi_{pm}$, where $u_{pm}$ are functions and not constant coefficients. For this purpose we define the following space of sequences of functions.

\begin{defn}
   If $(u_{pm}: p \ge -1, 0\le m\le2^p)$ is a family of affine functions of the form $u_{pm}: [t^0_{pm}, t^2_{pm}] \rightarrow \R^d$, %, where $u_{pm}(s) = a_{pm} + (s-t^0_{pm})b_{pm}$, then
   we set for $\alpha > 0$
   \begin{align*}
      \lVert (u_{pm})\rVert_{\A^\alpha} := \sup_{p,m} 2^{p\alpha} \lVert u_{pm}\rVert_\infty,
   \end{align*}
   where it is understood that $\lVert u_{pm} \rVert_\infty := \max_{t \in [t^0_{pm}, t^2_{pm}]} |u_{pm}(t)|$. The space $\A^\alpha := \A^\alpha(\R^d)$ is then defined as
   \[
      \A^\alpha := \left\{(u_{pm})_{p \ge -1, 0\le m\le2^p}: u_{pm}\in C([t^0_{pm}, t^2_{pm}], \R^d) \text{ is affine and } \lVert (u_{pm})\rVert_{\A^\alpha}<\infty \right\}.
   \]
\end{defn}

In Appendix~\ref{a:schauder with affine coefficients} we prove the following regularity estimate:

\begin{lem}\label{l:upm hoelder}
   Let $\alpha \in (0,2)$ and let $(u_{pm})\in \A^\alpha$. Then $\sum_{p,m} u_{pm} \varphi_{pm} \in \C^\alpha$, and
   \begin{align*}
      \Bigl\lVert \sum_{p,m} u_{pm} \varphi_{pm}\Bigr\rVert_\alpha \lesssim \lVert (u_{pm}) \rVert_{\A^\alpha}.
   \end{align*}
\end{lem}

Let us introduce a paraproduct in terms of Schauder functions.

\begin{lem}\label{l:paraproduct definition}
   Let $\beta \in (0,2)$, let $v \in C([0,1], \L(\R^d,\R^n))$, and $w \in \C^\beta(\R^d)$. Then
   \begin{align}\label{e:paraproduct definition}
      \pi_<(v,w) := \sum_{p=0}^\infty S_{p-1} v \Delta_p w \in \C^\beta(\R^n) \hspace{10pt} \text{and} \hspace{10pt} \lVert \pi_<(v,w) \rVert_\beta \lesssim \lVert v \rVert_\infty \lVert w \rVert_\beta.
   \end{align}
\end{lem}

\begin{proof}
   We have $\pi_<(v,w) = \sum_{p,m} u_{pm} \varphi_{pm}$ with $u_{pm} = (S_{p-1} v)|_{[t^0_{pm},t^2_{pm}]} w_{pm}$. For every $(p,m)$, the function $(S_{p-1} v)|_{[t^0_{pm},t^2_{pm}]}$ is the linear interpolation of $v$ between $t^0_{pm}$ and $t^2_{pm}$. As $\lVert (S_{p-1} v)|_{[t^0_{pm},t^2_{pm}]} w_{pm} \rVert_\infty \le  2^{-p\beta}\lVert v \rVert_\infty \lVert w \rVert_\beta$, the statement follows from Lemma~\ref{l:upm hoelder}.
\end{proof}

\begin{rmk}
   If $v \in \C^\alpha(\R)$ and $w \in \C^\beta(\R)$, we can decompose the product $vw$ into three components, $vw = \pi_<(v,w) + \pi_>(v,w) + \pi_\circ(v,w)$, where $\pi_>(v,w) :=  \pi_>(w,v)$ and $\pi_\circ(v,w):= \sum_p \Delta_p v \Delta_p w$, and we have the estimates
   \begin{align*}
      \lVert \pi_>(v,w) \rVert_\alpha \lesssim \lVert v \rVert_\alpha \lVert w \rVert_\infty, \qquad \text{and}\qquad \lVert \pi_\circ(v,w) \rVert_{\alpha+\beta} \lesssim \lVert v \rVert_\alpha \lVert w \rVert_\beta
   \end{align*}
   whenever $\alpha+\beta \in (0,2)$. However, we will not use this.
   %The estimate for $\pi_\circ$ only holds for $\alpha + \beta < 2$, and it is easy to show.%Since we will not use it, we omit the proof.
\end{rmk}

The paraproduct allows us to ``paralinearize'' nonlinear functions. We allow for a smoother perturbation, which will come in handy when constructing global in time solutions to SDEs.

\begin{prop}\label{p:paralinearization}
   Let $\alpha \in (0,1/2)$, $\beta \in (0,\alpha]$, let $v \in \C^\alpha(\R^d)$, $w \in \C^{\alpha+\beta}$, and $F \in C^{1+\beta/\alpha}_b(\R^d,\R)$. %Define
%   \begin{align*}
%      \pi_<( \dD F(v), v) := \sum_{|\eta|=1} \pi_< (\partial^\eta F(v), v^\eta).
%   \end{align*}
   Then% $F(v) - \pi_<(\dD F(v),v) \in \C^{2\alpha}$, and
   \begin{equation}\label{e:paralinearization estimate}
      \lVert F(v+w) - \pi_<(\dD F(v+w),v) \rVert_{\alpha + \beta} \lesssim \lVert F \rVert_{C^{1+\beta/\alpha}_b} (1 + \lVert v \rVert_\alpha)^{1+\beta/\alpha} (1 + \lVert w \rVert_{\alpha+\beta}).
   \end{equation}
   If $F \in C^{2+\beta/\alpha}_b$, then $F(v) - \pi_<(\dD F(v),v)$ depends on $v$ in a locally Lipschitz continuous way:
   \begin{align}\label{e:paralinearization lipschitz} \nonumber
      &\lVert F(v) - \pi_<(\dD F(v),v) - (F(u) - \pi_<(\dD F(u),u)) \rVert_{\alpha + \beta} \\
      &\hspace{160pt} \lesssim \lVert F \rVert_{C^{2+\beta/\alpha}_b} (1 + \lVert v \rVert_\alpha +  \lVert u \rVert_\alpha)^{1+\beta/\alpha} \lVert v - u\rVert_{\alpha}.
   \end{align}

\end{prop}

\begin{proof}
   First note that $\lVert F(v+w) \rVert_\infty \le \lVert F \rVert_\infty$, which implies the required estimate for $(p,m) = (-1,0)$ and $(p,m) = (0,0)$. For all other values of $(p,m)$ we apply a Taylor expansion:
   \begin{align*}
      (F(v+w))_{pm} %= \sum_{|\eta|=1} \partial^\eta F(v(t^1_{pm}) + w(t^1_{pm}))(v_{pm})^\eta + R_{pm}
      = \dD F(v(t^1_{pm}) + w(t^1_{pm}))v_{pm} + R_{pm},
   \end{align*}
   where $|R_{pm}| \lesssim 2^{- p (\alpha+\beta)} \lVert F\rVert_{C^{1+\beta/\alpha}_b} (\lVert v \rVert_\alpha^{1+\beta/\alpha} +  \lVert w \rVert_{\alpha+\beta})$. %Therefore, $F(v) = \sum_{pm} \dD F(v(t^1_{pm}))v_{pm} \varphi_{pm} + R$, with $R \in \C^{2\alpha}$ and $\lVert R \rVert_{2\alpha} \lesssim \lVert F\rVert_{C^2_b} \lVert v \rVert_\alpha^2$.
   Subtracting $\pi_<(\dD F(v),v)$ gives
   \begin{align*}
      &F(v+w) - \pi_<(\dD F(v+w),v) \\
      &\hspace{60pt}= \sum_{pm} [\dD F(v(t^1_{pm}) + w(t^1_{pm})) - (S_{p-1} \dD F(v+w))|_{[t^0_{pm}, t^2_{pm}]}] v_{pm} \varphi_{pm} + R.
   \end{align*}
   Now $(S_{p-1} \dD F(v+w))|_{[t^0_{pm}, t^2_{pm}]}$ is the linear interpolation of $\dD F(v+w)$ between $t^0_{pm}$ and $t^2_{pm}$, so according to Lemma~\ref{l:upm hoelder} it suffices to note that
   \begin{align*}
      &\lVert [\dD F(v(t^1_{pm})+ w(t^1_{pm})) - (S_{p-1} \dD F(v+w))|_{[t^0_{pm}, t^2_{pm}]}] v_{pm}\rVert_{\infty} \\
      &\hspace{50pt} \lesssim 2^{-p\beta} \lVert \dD F(v+w) \rVert_{C^\beta} 2^{-p\alpha} \lVert v \rVert_\alpha \lesssim 2^{-p(\alpha+\beta)} \lVert F \rVert_{C^{1+\beta/\alpha}_b} (1+\lVert v \rVert_\alpha + \lVert w \rVert_\alpha)^{\beta/\alpha} \lVert v \rVert_\alpha.
   \end{align*}
   The local Lipschitz continuity is shown in the same way.
\end{proof}

\begin{rmk}
   Since $v$ has compact support, it actually suffices to have $F \in C^{1+\beta/\alpha}$ without assuming boundedness. Of course, then the estimates in Proposition~\ref{p:paralinearization} have to be adapted.
\end{rmk}

\begin{rmk}\label{r:gubinelli controlled implies our controlled}
   The same proof shows that if $f$ is controlled by $v$ in the sense of Section~\ref{s:ciesielski}, i.e. $f_{s,t} = f^v(s) v_{s,t} + R_{s,t}$ with $f^v \in \C^\alpha$ and $|R_{s,t}|\le \lVert R\rVert_{2\alpha} |t-s|^{2\alpha}$, then $f - \pi_<(f^v,v) \in \C^{2\alpha}$.
\end{rmk}

\subsection{Young's integral and its different components}\label{s:young}

In this section we construct Young's integral using the Schauder expansion. If $v \in \C^\alpha$ and $w \in \C^\beta$, then we formally define
\begin{align*}
   \int_0^\cdot v(s) \dd w(s) := \sum_{p,m} \sum_{q,n} v_{pm} w_{qn} \int_0^\cdot \varphi_{pm}(s) \dd \varphi_{qn}(s) = \sum_{p,q} \int_0^\cdot \Delta_p v(s) \dd \Delta_q w(s).
\end{align*}
We show that this definition makes sense provided that $\alpha+\beta>1$, and we identify three components of the integral that behave quite differently. This will be our starting point towards an extension of the integral beyond the Young regime.

In a first step, let us calculate the iterated integrals of Schauder functions.
\begin{lem}\label{l:iterated schauder integrals1}
   Let $p > q \ge 0$. Then
   \begin{align}\label{e:iterated schauder integral p>q}
      \int_0^1 \varphi_{pm}(s) \dd \varphi_{qn}(s) = 2^{-p - 2} \chi_{qn}(t^0_{pm})
   \end{align}
   for all $m,n$. If $p = q$, then $\int_0^1 \varphi_{pm}(s) \dd \varphi_{pn}(s) = 0$, except if $p = q = 0$, in which case the integral is bounded by 1. If $0 \le p < q$, then for all $(m,n)$ we have
   \begin{align}\label{e:iterated schauder integral q<p}
      \int_0^1 \varphi_{pm}(s) \dd \varphi_{qn}(s) = - 2^{-q - 2} \chi_{pm}\left(t^0_{qn}\right).
   \end{align}
   If $p=-1$, then the integral is bounded by 1.
\end{lem}

\begin{proof}
   The cases $p = q$ and $p=-1$ are easy, so let $p > q \ge 0$. Since $\chi_{qn} \equiv \chi_{qn}(t^0_{pm})$ on the support of $\varphi_{pm}$, we have
   \begin{align*}
      \int_0^1 \varphi_{pm}(s) \dd \varphi_{qn}(s) = \chi_{qn}(t^0_{pm}) \int_0^1 \varphi_{pm}(s) \dd s = \chi_{qn}(t^0_{pm}) 2^{-p-2}.
   \end{align*}
   If $0 \le p < q$, then integration by parts and \eqref{e:iterated schauder integral p>q} imply \eqref{e:iterated schauder integral q<p}.
\end{proof}

Next we estimate the coefficients of iterated integrals in the Schauder basis.

\begin{lem}\label{l:schauder coefficients of iterated integrals}
   Let $i,p\ge -1$, $q \ge 0$, $0\le j \le 2^i$, $0\le m \le 2^p$, $0\le n \le 2^q$. Then
   \begin{align}\label{e:schauder coefficients of iterated integrals good}
      2^{-i} \Big|\Big\langle \chi_{ij}, \dd\Big(\int_0^\cdot\varphi_{pm} \chi_{qn}\dd s\Big)\Big\rangle\Big| \le 2^{-2(i \vee p \vee q) + p + q},
   \end{align}
   except if $p<q=i$. In this case we only have the worse estimate
   \begin{align}\label{e:schauder coefficients of iterated integrals bad}
      2^{-i} \Big|\Big\langle \chi_{ij}, \dd\Big(\int_0^\cdot\varphi_{pm} \chi_{qn}\dd s\Big)\Big\rangle\Big| \le 1.
   \end{align}
\end{lem}

\begin{proof}
   We have $\langle \chi_{-10}, \dd(\int_0^\cdot \varphi_{pm} \chi_{qn}\dd s)\rangle = 0$ for all $(p,m)$ and $(q,n)$. So let $i \ge 0$. If $i < p \vee q$, then $\chi_{ij}$ is constant on the support of $\varphi_{pm}\chi_{qn}$, and therefore Lemma~\ref{l:iterated schauder integrals1} gives
   \[
      2^{-i} \left|\langle \chi_{ij},\varphi_{pm} \chi_{qn}\rangle\right| \le \left|\langle \varphi_{pm}, \chi_{qn}\rangle\right| \le 2^{ p + q -2(p\vee q)} = 2^{-2(i \vee p \vee q) + p + q}.
   \]

   Now let $i > q$. Then $\chi_{qn}$ is constant on the support of $\chi_{ij}$, and therefore another application of Lemma~\ref{l:iterated schauder integrals1} implies that
   \[
      2^{-i} \left|\langle \chi_{ij}, \varphi_{pm}\chi_{qn}\rangle\right| \le 2^{-i} 2^q 2^{p+i -2(p\vee i)} = 2^{-2(i \vee p \vee q) + p + q}.
   \]

   The only remaining case is $i=q \ge p$, in which%. If $p=i=q$, then
%   \begin{align*}
%      2^{-i} \left|\langle \chi_{ij},\varphi_{pm} \chi_{qn}\rangle\right| \le 2^{p} \int_0^1 \varphi_{pm}(s) \dd s \le 2^{2p} 2^{-p} = 2^{-2(i\vee p \vee q) + p + q + i}.
%   \end{align*}
%   Otherwise, if $i = q > p$, then
   \[
      2^{-i} \left|\langle \chi_{ij},\varphi_{pm} \chi_{qn}\rangle\right| \le 2^{i} \int_{t^0_{ij}}^{t^2_{ij}} \varphi_{pm}(s) \dd s \le \lVert \varphi_{pm} \rVert_\infty \le 1.
   \]
\end{proof}

%We use this result to estimate the iterated integrals of Schauder blocks.

\begin{cor}\label{c:schauder blocks}
   Let $i, p\ge -1$ and $q \ge 0$. Let $v \in C([0,1],\L(\R^d,\R^n))$ and $w \in C([0,1],\R^d)$. Then
   \begin{align}\label{e:schauder blocks good}
      \Big\lVert \Delta_i\Big(\int_0^\cdot \Delta_p v(s) \dd \Delta_q w(s)\Big)\Big\rVert_\infty \lesssim 2^{-(i\vee p\vee q) - i+p+q} \lVert \Delta_p v \rVert_\infty \lVert \Delta_q w \rVert_\infty,
   \end{align}
   except if $i=q>p$. In this case we only have the worse estimate
   \begin{align}\label{e:schauder blocks bad}
      \Big \lVert \Delta_i\Big(\int_0^\cdot \Delta_p v (s) \dd \Delta_q w(s)\Big)\Big\rVert_\infty \lesssim \lVert \Delta_p v \rVert_\infty \lVert \Delta_q w \rVert_\infty.
   \end{align}
\end{cor}

\begin{proof}
   The case $i = -1$ is easy, so let $i \ge 0$. We have
   \begin{align*}
      \Delta_i\Big(\int_0^\cdot \Delta_p v(s) \dd \Delta_q w(s)\Big) = \sum_{j,m,n} v_{pm} w_{qn} \langle 2^{-i} \chi_{ij}, \varphi_{pm} \chi_{qn}\rangle \varphi_{ij}.
   \end{align*}
   For fixed $j$, there are at most $2^{(i\vee p\vee q) - i}$ non-vanishing terms in the double sum. %Furthermore, we have $|v_{pm}|\lesssim \lVert \Delta_p v \rVert_\infty$ and similarly for $|w_{qn}|$.
   Hence, we obtain from Lemma~\ref{l:schauder coefficients of iterated integrals} that
   \begin{align*}
      \Big\lVert \sum_{m,n} v_{pm} w_{qn} \langle 2^{-i} \chi_{ij}, \varphi_{pm} \chi_{qn}\rangle \varphi_{ij}\Big\rVert_\infty & \lesssim 2^{(i\vee p\vee q) - i} \lVert \Delta_p v \rVert_\infty \lVert \Delta_q w \rVert_\infty  (2^{-2(i\vee p \vee q) + p + q} + \1_{i=q>p}) \\
      & = (2^{-(i\vee p\vee q) - i + p + q} + \1_{i=q>p}) \lVert \Delta_p v \rVert_\infty \lVert \Delta_q w \rVert_\infty.
   \end{align*}
%   except if $i=q>p$. In that case Lemma~\ref{l:schauder coefficients of iterated integrals} yields
%   \begin{align*}
%      \left\lVert \sum_{m,n} v_{pm} w_{qn} \langle 2^{-i} \chi_{ij}, \varphi_{pm} \chi_{qn}\rangle \varphi_{ij}\right\rVert_\infty & \lesssim 2^{i-i}\lVert \Delta_p v \rVert_\infty \lVert \Delta_q w \rVert_\infty = \lVert \Delta_p v \rVert_\infty \lVert \Delta_q w\rVert_\infty.
%   \end{align*}
\end{proof}

\begin{cor}\label{c:schauder blocks product}
   Let $i,p,q \ge -1$. Let $v \in C([0,1],\L(\R^d,\R^n))$ and $w \in C([0,1],\R^d)$. Then for $p \vee q \le i$ we have
   \begin{align}\label{e:schauder blocks product good}
      \left\lVert \Delta_i\left(\Delta_p v \Delta_q w\right)\right\rVert_\infty \lesssim 2^{-(i\vee p\vee q) - i+p+q} \lVert \Delta_p v \rVert_\infty \lVert \Delta_q w \rVert_\infty,
   \end{align}
   except if $i=q>p$ or $i=p>q$, in which case we only have the worse estimate
   \begin{align}\label{e:schauder blocks product bad}
      \left \lVert \Delta_i(\Delta_p v \Delta_q w)\right\rVert_\infty \lesssim \lVert \Delta_p v \rVert_\infty \lVert \Delta_q w \rVert_\infty.
   \end{align}
   If $p > i$ or $q>i$, then $\Delta_i(\Delta_p v \Delta_q w) \equiv 0$.
\end{cor}

\begin{proof}
   The case $p=-1$ or $q=-1$ is easy. Otherwise we apply integration by parts and note that the estimates \eqref{e:schauder blocks good} and \eqref{e:schauder blocks bad} are symmetric in $p$ and $q$. If for example $p>i$, then $\Delta_p(v)(t^k_{ij}) = 0$ for all $k,j$, which implies that $\Delta_i (\Delta_p v \Delta_q w) = 0$.
\end{proof}

The estimates \eqref{e:schauder blocks good} and \eqref{e:schauder blocks bad} allow us to identify different components of the integral $\int_0^\cdot v(s) \dd w(s)$. More precisely, \eqref{e:schauder blocks bad} indicates that the series $\sum_{p<q} \int_0^\cdot \Delta_p v(s) \dd \Delta_q w(s)$ is rougher than the remainder $\sum_{p \ge q} \int_0^\cdot \Delta_p v(s) \dd \Delta_q w(s)$. Integration by parts gives%to $\int_0^\cdot \Delta_p v(s) \dd \Delta_q w(s)$, then we obtain
\[
   \sum_{p<q} \int_0^\cdot \Delta_p v(s) \dd \Delta_q w(s) = \pi_<(v,w) - \sum_{p<q} \sum_{m,n} v_{pm} w_{qn} \int_0^\cdot \varphi_{qn}(s) \dd \varphi_{pm}(s).
\]
This motivates us to decompose the integral into three components, namely
\begin{align*}
    \sum_{p,q} \int_0^\cdot \Delta_p v(s) \dd \Delta_q w(s) = L(v,w) + S(v,w) + \pi_<(v,w).
\end{align*}
Here $L$ is defined as the antisymmetric \emph{L\'evy area} (we will justify the name below by showing that $L$ is closely related to the L\'evy area of certain dyadic martingales):
\begin{align*}
   L(v,w) :=\,& \sum_{p>q} \sum_{m,n} (v_{pm} w_{qn} - v_{qn} w_{pm}) \int_0^\cdot \varphi_{pm} \dd \varphi_{qn}\\
   =\,& \sum_{p} \left(\int_0^\cdot \Delta_p v \dd S_{p-1} w - \int_0^\cdot \dd (S_{p-1} v) \Delta_{p} w\right).
\end{align*}
The \emph{symmetric part} $S$ is defined as
\begin{align*}
   S(v,w) :=\, & \sum_{m,n \le 1} v_{0m} w_{0n} \int_0^\cdot \varphi_{0m} \dd \varphi_{0n} + \sum_{p\ge 1} \sum_m v_{pm} w_{pm} \int_0^\cdot \varphi_{pm} \dd \varphi_{pm} \\
   =\,& \sum_{m,n\le 1} v_{0m} w_{0n} \int_0^\cdot \varphi_{0m} \dd \varphi_{0n} + \frac{1}{2} \sum_{p\ge 1} \Delta_p v \Delta_p w,
\end{align*}
and $\pi_<$ is the paraproduct defined in \eqref{e:paraproduct definition}. As we observed in Lemma~\ref{l:paraproduct definition}, $\pi_<(v,w)$ is always well defined, and it inherits the regularity of $w$. Let us study $S$ and $L$.

\begin{lem}\label{l:Levy area regularity}
   Let $\alpha, \beta \in (0,1)$ be such that $\alpha + \beta > 1$. Then $L$ is a bounded bilinear operator from $\C^\alpha \times \C^\beta$ to $\C^{\alpha+\beta}$.
%   , and let $v \in \C^\alpha$ and $w \in \C^\beta$. Then $L(v,w)$ is well defined and in $\C^{\alpha+\beta}$, and moreover
%   \begin{align*}
%      \lVert L(v,w) \rVert_{\alpha + \beta} \lesssim_{\alpha + \beta} \lVert v \rVert_\alpha \lVert w \rVert_\beta.
%   \end{align*}
\end{lem}

\begin{proof}
   We only argue for $\sum_{p} \int_0^\cdot \Delta_p v \dd S_{p-1} w$, the term $- \int_0^\cdot \dd (S_{p-1} v) \Delta_{p} w$ can be treated with the same arguments. Corollary~\ref{c:schauder blocks} (more precisely \eqref{e:schauder blocks good}) implies that
   \begin{align*}
      &\Big\lVert \sum_p \Delta_i \Big(\int_0^\cdot \Delta_p v \dd  S_{p-1} w\Big) \Big\rVert_\infty \\
      &\hspace{50pt} \le \sum_{p\le i} \sum_{q<p} \Big\lVert \Delta_i \Big(\int_0^\cdot \Delta_p v \dd  \Delta_q w \Big)\Big\rVert_\infty + \sum_{p> i} \sum_{q<p} \Big\lVert \Delta_i \Big( \int_0^\cdot \Delta_p v \dd  \Delta_q w \Big)\Big\rVert_\infty\\
      &\hspace{50pt} \le \bigg(\sum_{p\le i} \sum_{q<p} 2^{-2i + p + q} 2^{-p\alpha}\lVert v \rVert_\alpha 2^{-q\beta} \lVert w \rVert_\beta +  \sum_{p> i} \sum_{q<p} 2^{- i + q} 2^{-p\alpha} \lVert v \rVert_\alpha 2^{-q\beta} \lVert w \rVert_\beta  \bigg) \\
      &\hspace{50pt} \lesssim_{\alpha + \beta} 2^{-i(\alpha+\beta)} \lVert v \rVert_\alpha \lVert w \rVert_\beta,
   \end{align*}
   where we used $1-\alpha < 0$ and $1-\beta<0$ and for the second series we also used that $\alpha+\beta>1$.
\end{proof}

Unlike the L\'evy area $L$, the symmetric part $S$ is always well defined. It is also smooth.

\begin{lem}\label{l:symmetric part}
   Let $\alpha,\beta \in (0,1)$. Then $S$ is a bounded bilinear operator from $\C^\alpha \times \C^\beta$ to $\C^{\alpha+\beta}$.
%   , and let $v \in \C^\alpha$ and $w \in \C^\beta$. Then $S(v,w) \in \C^{\alpha+\beta}$, and
%   \begin{align*}
%      \lVert S(v,w) \rVert_{\alpha+\beta} \lesssim \lVert v \rVert_\alpha \lVert w \rVert_\beta.
%   \end{align*}
\end{lem}

\begin{proof}
   This is shown using the same arguments as in the proof of Lemma~\ref{l:Levy area regularity}.
\end{proof}

In conclusion, the integral consists of three components. The L\'evy area $L(v,w)$ is only defined if $\alpha + \beta>1$, but then it is smooth. The symmetric part $S(v,w)$ is always defined and smooth. And the paraproduct $\pi_<(v,w)$ is always defined, but it is rougher than the other components. To summarize:

\begin{thm}[Young's integral]\label{t:young integral}
   Let $\alpha, \beta \in (0,1)$ be such that $\alpha + \beta > 1$, and let $v \in \C^\alpha$ and $w \in \C^\beta$. Then the integral
   \begin{align*}
      I(v,\dd w) := \sum_{p,q} \int_0^\cdot \Delta_p v \dd \Delta_q w = L(v,w) + S(v,w) + \pi_<(v,w) \in \C^\beta
   \end{align*}
   satisfies $\lVert I(v,\dd w) \rVert_\beta \lesssim \lVert v \rVert_\alpha \lVert w \rVert_\beta$ and
   \begin{align}\label{e:Young controlled}
      \lVert I(v,\dd w) - \pi_<(v,w) \rVert_{\alpha+\beta} \lesssim  \lVert v \rVert_\alpha \lVert w \rVert_\beta.
   \end{align}
\end{thm}

\subsubsection*{L\'evy area and dyadic martingales}

Here we show that the L\'evy area $L(v,w)(1)$ can be expressed in terms of the L\'evy area of suitable dyadic martingales. To simplify notation, we assume that $v(0) = w(0) = 0$, so that we do not have to bother with the components $v_{-10}$ and $w_{-10}$.

We define a filtration $(\F_n)_{n \ge 0}$ on $[0,1]$ by setting
\begin{align*}
   \F_n = \sigma(\chi_{p m}: 0 \le p \le n, 0 \le m \le 2^p),
\end{align*}
we set $\F = \bigvee_n \F_n$, and we consider the Lebesgue measure on $([0,1], \F)$. On this space, the process $M_n = \sum_{p=0}^n \sum_{m=0}^{2^p} \chi_{pm}$, $n \in \N$, is a martingale. For any continuous function $v:[0,1] \rightarrow \R$ with $v(0) = 0$, the process
\begin{align*}
   M^v_n = \sum_{p=0}^n \sum_{m=0}^{2^p} \langle 2^{-p} \chi_{pm}, \dd v\rangle \chi_{pm} = \sum_{p=0}^n \sum_{m=0}^{2^p} v_{pm} \chi_{pm} = \partial_t S_n v,
\end{align*}
$n \in \N$, is a martingale transform of $M$, and therefore a martingale as well. Since it will be convenient later, we also define $\F_{-1} = \{\emptyset, [0,1]\}$ and $M_{-1}^v = 0$ for every $v$.

Assume now that $v$ and $w$ are continuous real-valued functions with $v(0) = w(0) = 0$, and that the L\'evy area $L(v,w)(1)$ exists. Then it is given by
\begin{align*}
   L(v,w)(1) & = \sum_{p=0}^\infty \sum_{q=0}^{p-1} \sum_{m,n} (v_{pm} w_{qn} - v_{qn} w_{pm}) \int_0^1 \varphi_{pm}(s) \chi_{qn}(s) \dd s \\
%   & = \sum_{p=0}^\infty \sum_{q=0}^{p-1} \sum_{m,n} (v_{pm} w_{qn} - v_{qn} w_{pm}) \chi_{qn}(t^0_{pm}) \int_0^1 \varphi_{pm}(s) \dd s \\
   & = \sum_{p=0}^\infty \sum_{q=0}^{p-1} \sum_{m,n} (v_{pm} w_{qn} - v_{qn} w_{pm}) 2^p \int_0^1 \chi_{qn}(s) 1_{[t^0_{pm}, t^2_{pm})}(s) \dd s \langle \varphi_{pm},1 \rangle \\
   & = \sum_{p=0}^\infty \sum_{q=0}^{p-1} \sum_{m,n} (v_{pm} w_{qn} - v_{qn} w_{pm}) 2^{-p} \int_0^1 \chi_{qn}(s) \chi_{pm}^2(s) \dd s 2^{-p-2}\\
   & = \sum_{p=0}^\infty \sum_{q=0}^{p-1} 2^{-2p-2} \int_0^1 \sum_{m,n} \sum_{m'} (v_{pm} w_{qn} - v_{qn} w_{pm}) \chi_{qn}(s) \chi_{pm}(s) \chi_{pm'}(s) \dd s,
\end{align*}
where %in the third line we used that $\langle \varphi_{pm}, 1\rangle = 2^{-p-2}$ for all $p \ge 1$, and
in the last step we used that $\chi_{pm}$ and $\chi_{pm'}$ have disjoint support for $m \neq m'$. The $p$--th \emph{Rademacher function} (or \emph{square wave}) is defined for $p \ge 1$ as
\begin{align*}
   r_p(t) := \sum_{m'=1}^{2^p} 2^{-p} \chi_{pm'}(t).
\end{align*}
The martingale associated to the Rademacher functions is given by $R_0 := 0$ and $R_p := \sum_{k=1}^p r_k$ for $p \ge 1$. Let us write $\Delta M^v_p = M^v_p - M^v_{p-1}$ and similarly for $M^w$ and $R$ and all other discrete time processes that arise. This notation somewhat clashes with the expression $\Delta_p v$ for the dyadic blocks of $v$, but we will only use it in the following lines, where we do not directly work with dyadic blocks. The quadratic covariation of two dyadic martingales is defined as $[M, N]_n := \sum_{k=0}^n \Delta M_k \Delta N_k$, and the discrete time stochastic integral is defined as $(M\cdot N)_n := \sum_{k=0}^n M_{k-1} \Delta N_k$. Writing $E(\cdot)$ for the integral $\int_0^1 \cdot \dd s$, we obtain
\begin{align*}
   L(v,w)(1) & = \sum_{p=0}^\infty \sum_{q=0}^{p-1} 2^{-p-2} E\left(\Delta M^v_p \Delta M^w_q \Delta R_p - \Delta M^v_q \Delta M^w_p \Delta R_p \right) \\
   & = \sum_{p=0}^\infty 2^{-p-2} E\left(\left(M^w_{p-1} \Delta M^v_p - M^v_{p-1} \Delta M^w_p \right) \Delta R_p \right)\\
   & = \sum_{p=0}^\infty 2^{-p-2} E\left(\Delta \left[M^w \cdot M^v - M^v \cdot M^w, R\right]_p \right).
\end{align*}
Hence, $L(v,w)(1)$ is closely related to the L\'evy area $1/2(M^w \cdot M^v - M^v \cdot M^w)$ of the dyadic martingale $(M^v, M^w)$.

\section{Paracontrolled paths and pathwise integration beyond Young}\label{s:schauder rough path integral}

In this section we construct a rough path integral in terms of Schauder functions.

\subsection{Paracontrolled paths}

We observed in Section~\ref{s:paradifferential calculus} that for $w \in \C^\alpha$ and $F \in C^{1+\beta/\alpha}_b$ we have $F(w) - \pi_<(\dD F(w),w) \in \C^{\alpha+\beta}$. In Section~\ref{s:young} we observed that if $v \in \C^\alpha$, $w\in \C^\beta$ and $\alpha + \beta > 1$, then the Young integral $I(v,\dd w)$ satisfies $I(v,\dd w) - \pi_<(v,w) \in \C^{\alpha + \beta}$. Hence, in both cases the function under consideration can be written as $\pi_<(f^w,w)$ for a suitable $f^w$, plus a smooth remainder. We make this our definition of paracontrolled paths:
\begin{defn}
   Let $\alpha > 0$ and $v \in \C^\alpha(\R^d)$. For $\beta \in (0,\alpha]$ we define
   \[
      \mathcal{D}^{\beta}_v := \D^\beta_v(\R^n) := \left\{(f,f^v) \in \C^\alpha(\R^n) \times \C^\beta(\L(\R^d,\R^n)): f^\sharp = f - \pi_<(f^v,v)\in \C^{\alpha+\beta}(\R^n)\right\}.
   \]
   If $(f,f^v) \in \D^\beta_v$, then $f$ is called \emph{paracontrolled} by $v$. The function $f^v$ is called the \emph{derivative} of $f$ with respect to $v$. Abusing notation, we write $f \in \D^\beta_v$ when it is clear from the context what the derivative $f^v$ is supposed to be. We equip $\D^\beta_v$ with the norm
   \begin{align*}
      \lVert f \rVert_{v,\beta} := \lVert f^v \rVert_\beta + \lVert f^\sharp\rVert_{\alpha+\beta}.
   \end{align*}
   If $v \in \C^\alpha$ and $(\tilde f, \tilde f^{\tilde v}) \in \D^\beta_{\tilde v}$, then we also write
   \[
      d_{\D^\beta}(f,\tilde f) := \lVert f^v - \tilde f^{\tilde v} \rVert_\beta + \lVert f^\sharp - \tilde f^\sharp \rVert_{\alpha+\beta}.
   \]
\end{defn}

%\begin{rmk}
%   In general the derivative $f^v$ is not uniquely determined by $f$ and $v$. If for example $v \in \C^{\alpha + \beta}$, then $0 \in \D^\beta_v$ and every $f^v \in \C^\beta$ can be taken as its derivative. So the correct definition would be $(f,f^v) \in \D^\beta_v$ and $\lVert(f,f^v)\rVert_{v,\beta} = \lVert f \rVert_\alpha + \lVert f^v \rVert_\beta + \lVert f^\sharp\rVert_{\alpha+\beta}$. But usually there will be no confusion about the derivative that we have in mind and we will continue writing $f\in \D^\beta_v$ and $\lVert f\rVert_{v,\beta}$.
%\end{rmk}

\begin{ex}
   Let $\alpha \in (0,1)$ and $v \in \C^\alpha$. Then Proposition~\ref{p:paralinearization} shows that $F(v) \in \D^\beta_v$ for every $F \in C^{1+\beta/\alpha}_b$, with derivative $\dD F(v)$.
\end{ex}

\begin{ex}
   Let $\alpha + \beta >1$ and $v\in \C^\alpha$, $w \in \C^\beta$. Then by~\eqref{e:Young controlled}, the Young integral $I(v,\dd w)$ is in $\D^\alpha_w$, with derivative $v$.
\end{ex}

\begin{ex}\label{ex:controlled old vs new}
   If $\alpha + \beta < 1$ and $v \in \C^\alpha$, then $(f,f^v) \in \D^\beta_v$ if and only if $|f_{s,t} - f^v_s v_{s,t}| \lesssim |t-s|^{\alpha+\beta}$ and in that case
   \[
      \lVert f^v\rVert_\infty + \sup_{s\neq t} \frac{| f^v_{s,t} |}{|t-s|^\beta} + \sup_{s \neq t}\frac{|f_{s,t} - f^v_s v_{s,t}|}{|t-s|^{\alpha+\beta}} \lesssim \|f\|_{v,\beta}(1+\lVert v \rVert_\alpha).
   \]
   Indeed we have $|f^v_s v_{s,t} - \pi_<(f^v,v)_{s,t}| \lesssim |t-s|^{\alpha+\beta} \lVert f^v\rVert_\beta \lVert v \rVert_\alpha$, which can be shown using similar arguments as for Lemma~B.2 in~\cite{Gubinelli2012}.
   In other words, for $\alpha \in (0,1/2)$ the space $\D^\alpha_v$ coincides with the space of controlled paths defined in Section~\ref{s:rough paths}.
\end{ex}

The following commutator estimate, the analog of Theorem~2.3 of~\cite{Bony1981} in our setting, will be useful for establishing some stability properties of~$\D^\beta_v$.

\begin{lem}\label{l:commutator 2}
   Let $\alpha, \beta \in (0,1)$, and let $u\in C([0,1],\L(\R^n;\R^m))$, $v\in \C^\alpha(\L(\R^d;\R^n))$, and $w \in \C^\beta(\R^d)$. Then
   \begin{align*}
      \lVert \pi_<(u, \pi_<(v,w)) - \pi_<(uv, w)\rVert_{\alpha + \beta} \lesssim \lVert u\rVert_\infty \lVert v \rVert_\alpha \lVert w \rVert_\beta.
   \end{align*}
\end{lem}

\begin{proof}
   We have
   \begin{align*}
      \pi_<(u, \pi_<(v,w)) - \pi_<(uv, w) & = \sum_{p,m} (S_{p-1}u (\pi_<(v,w))_{pm} - S_{p-1}(uv) w_{pm})\varphi_{pm} %\\
      %& = \sum_{p,m} \left(S_{p-1}x \left(\sum_q S_{q-1}y \Delta_q z\right)_{pm} - S_{p-1}(xy) z_{pm}\right)\varphi_{pm},
   \end{align*}
   and $[S_{p-1}u (\pi_<(v,w))_{pm} - S_{p-1}(uv) w_{pm}]|_{[t^0_{pm},t^2_{pm}]}$ is affine. By Lemma \ref{l:upm hoelder} it suffices to control $\lVert[S_{p-1}u (\pi_<(v,w))_{pm} - S_{p-1}(uv) w_{pm}]|_{[t^0_{pm},t^2_{pm}]}\rVert_\infty$.

%   For $(p,m) = (-1,0)$ or $(p,m) = (0,0)$ we simply have
%   \begin{align*}
%      \left\lVert S_{p-1}x (\pi_<(y,z))_{pm} - S_{p-1}(xy) z_{pm}\right\rVert_\infty \lesssim \lVert x \rVert_\infty \lVert y \rVert_\infty \lVert z \rVert_\infty \lesssim \lVert x\rVert_\infty \lVert y \rVert_\alpha \lVert z \rVert_\beta.
%   \end{align*}

   The cases $(p,m) = (-1,0)$ and $(p,m) = (0,0)$ are easy, so let $p \ge 0$ and $m \ge 1$. For $r<q<p$ we denote by $m_q$ and $m_r$ the unique index in generation $q$ and $r$ respectively for which $\chi_{pm} \varphi_{q m_q} \not \equiv 0$ and similarly for $r$. We apply Lemma \ref{l:schauder coefficients of iterated integrals} to obtain for $q<p$
   \begin{align*}
      |(S_{q-1}v \Delta_q w)_{pm}| & = \Big|\sum_{r<q} v_{rm_r} w_{qm_q} 2^{-p} \langle \chi_{pm}, \dd(\varphi_{rm_r} \varphi_{qm_q})\rangle \Big|\\
      & = \Big|\sum_{r<q} v_{rm_r} w_{qm_q} 2^{-p} \langle \chi_{pm}, \chi_{rm_r} \varphi_{qm_q} + \varphi_{rm_r} \chi_{qm_q} \rangle\Big| \\
      & \le \lVert v \rVert_\alpha \lVert w \rVert_\beta \sum_{r<q} 2^{-r\alpha} 2^{-q\beta} 2^{-p}2^{-2p+r+p+q} \lesssim 2^{-2p+q(2-\alpha-\beta)} \lVert v \rVert_\alpha \lVert w \rVert_\beta.
   \end{align*}
   Hence
   \[
      \Big\lVert \Big( S_{p-1}u \sum_{q<p} (S_{q-1} v \Delta_q w)_{pm} \Big) \Big|_{[t^0_{pm},t^2_{pm}]}\Big\rVert_{\infty} \lesssim \lVert u \rVert_\infty \lVert v \rVert_\alpha \lVert w \rVert_\beta 2^{-p(\alpha+\beta)}.
   \]
   If $p<q$, then $\Delta_q w(t^k_{pm}) = 0$ for all $k$ and $m$, and therefore $(S_{q-1}v \Delta_q w)_{pm}=0$, so that it only remains to bound $\|[S_{p-1}u (S_{p-1}v \Delta_p w)_{pm} - S_{p-1}(uv) w_{pm}]|_{t^0_{pm}, t^2_{pm}]}\|_\infty$. We have $\Delta_p w(t^0_{pm}) = \Delta_p w(t^2_{pm})=0$ and $\Delta_p w(t^1_{pm}) = w_{pm}/2$. On $[t^0_{pm}, t^2_{pm}]$, the function $S_{p-1} v$ is given by the linear interpolation of $v(t^0_{pm})$ and $v(t^2_{pm})$, and therefore $(S_{p-1}v \Delta_p w)_{pm} = \frac{1}{2}(v(t^0_{pm}) + v(t^2_{pm}))w_{pm}$, leading to
   \begin{align*}
      &\lVert [S_{p-1}u (S_{p-1}v \Delta_p w)_{pm} - S_{p-1}(uv) w_{pm}]|_{[t^0_{pm},t^2_{pm}]}\rVert_{\infty}\\
      &\hspace{80pt} \le |w_{pm}|\times  \Big\lVert\Big[ \Big(u(t^0_{pm})+\frac{\cdot-t^0_{pm}}{t^2_{pm}-t^0_{pm}}u_{t^0_{pm},t^2_{pm}}\Big)\frac{v(t^0_{pm}) + v(t^2_{pm})}{2} \\
      &\hspace{160pt} - \Big((uv)(t^0_{pm}) + \frac{\cdot-t^0_{pm}}{t^2_{pm}-t^0_{pm}}(uv)_{t^0_{pm},t^2_{pm}}\Big)\Big]\Big|_{[t^0_{pm},t^2_{pm}]}\Big\rVert_{\infty}  \\
      &\hspace{80pt} \lesssim \lVert u\rVert_{\infty} \lVert v \rVert_\alpha \lVert w \rVert_\beta 2^{-p(\alpha+\beta)},
   \end{align*}
   where the last step follows by rebracketing.
\end{proof}

As a consequence, we can show that paracontrolled paths are stable under the application of smooth functions.

\begin{cor}\label{c:controlled under smooth}
   Let $\alpha \in (0,1)$, $\beta \in (0,\alpha]$, $v \in \C^\alpha$, and $f \in \D^\beta_v$ with derivative $f^v$. Let $F \in C^{1+\beta/\alpha}_b$. Then $F(f) \in \D^\beta_v$ with derivative $\dD F(f) f^v$, and
   \begin{align*}
      \lVert F(f)\rVert_{v,\beta} \lesssim \lVert F\rVert_{C^{1+\beta/\alpha}_b} (1 + \lVert v \rVert_\alpha)^{1+\beta/\alpha} (1+\lVert f \rVert_{v,\beta}) (1 + \lVert f^v \rVert_\infty)^{1+\beta/\alpha}.
   \end{align*}
   Moreover, there exists a polynomial $P$ which satisfies for all $F \in C^{2+\beta/\alpha}_b$, $\tilde v \in \C^\alpha$, $\tilde f \in \D^\beta_{\tilde v}$, and
   \[
      M = \max\{\lVert v \rVert_\alpha, \lVert \tilde v \rVert_\alpha, \lVert f \rVert_{v,\beta}, \lVert \tilde f \rVert_{\tilde v,\beta}\}
   \]
   the bound
   \[
      d_{\D^\beta}(F(f), F(\tilde f)) \le P(M) \lVert F\rVert_{C^{2+\beta/\alpha}_b}(d_{\D^\beta}(f,\tilde f) + \lVert u - \tilde u\rVert_\alpha).
   \]
\end{cor}

\begin{proof}
   The estimate for $\|\dD F(f) f^v\|_\beta$ is straightforward. For the remainder we apply Proposition~\ref{p:paralinearization} and Lemma~\ref{l:commutator 2} to obtain
   \begin{align*}
      \|F(f)^\sharp\|_{\alpha+\beta} & \le \|F(f) - \pi_<(\dD F(f), f)\|_{\alpha+\beta} + \|\pi_<(\dD F(f), f^\sharp)\|_{\alpha+\beta} \\
      &\quad + \|\pi_<(\dD F(f), \pi_<(f^v,v)) - \pi_<(\dD F(f) f^v, v) \|_{\alpha+\beta} \\
      &\lesssim \lVert F \rVert_{C^{1+\beta/\alpha}_b} (1 + \lVert \pi_<(f^v,v)\rVert_\alpha)^{1+\beta/\alpha}(1 + \lVert f^\sharp \rVert_{\alpha+\beta})  \\
      &\quad + \lVert F \rVert_{C^1_b} \lVert f \rVert_{v,\beta} + \lVert F \rVert_{C^1_b} \lVert f^v \rVert_{\beta} \lVert v \rVert_\alpha \\
      &\lesssim \lVert F\rVert_{C^{1+\beta/\alpha}_b} (1+\lVert f^v \rVert_{\infty})^{1+\beta/\alpha} (1 + \lVert v \rVert_\alpha)^{1+\beta/\alpha}(1+\lVert f \rVert_{v,\beta}).
   \end{align*}
   The difference $F(f) - F(\tilde f)$ is treated in the same way.
\end{proof}

When solving differential equations it will be crucial to have a bound which is linear in $\lVert f \rVert_{v,\beta}$. The superlinear dependence on $\lVert f^v \rVert_\infty$ will not pose any problem as we will always have $f^v = F(\tilde f)$ for some suitable $\tilde f$, so that for bounded $F$ we get $\lVert F(f)\rVert_{v,\beta} \lesssim_{F,v} 1 + \lVert f \rVert_{v,\beta}$.

\subsection{A basic commutator estimate}

Here we prove the commutator estimate which will be the main ingredient in the construction of the integral $I(f,\dd g)$, where $f$ is paracontrolled by $v$ and $g$ is paracontrolled by $w$, and where we assume that the integral $I(v,\dd w)$ exists.

\begin{prop}\label{p:commutator 1}
   Let $\alpha, \beta, \gamma \in (0,1)$, and assume that $\alpha+\beta+\gamma>1$ and $\beta+\gamma < 1$. Let $f\in \C^\alpha$, $v\in \C^\beta$, and $w \in \C^\gamma$. Then the ``commutator''
   \begin{align}\label{e:commutator 1 def}
      C(f,v,w) &:= L(\pi_<(f, v),w) - I(f,\dd L(v,w))\\ \nonumber
      & := \lim_{N\rightarrow \infty} [L(S_N(\pi_<(f, v)), S_N w) - I(f,\dd L(S_N v,S_N w))] \\ \nonumber
      &\, = \lim_{N\rightarrow \infty} \sum_{p\le N} \sum_{q<p} \Biggl[ \int_0^\cdot \Delta_p (\pi_<(f,v))(s) \dd \Delta_{q} w(s) - \int_0^\cdot \dd (\Delta_{q} (\pi_<(f,v)))(s) \Delta_{p} w(s) \\ \nonumber
      &\hspace{90pt} - \Bigl(\int_0^\cdot f(s) \Delta_p v(s) \dd \Delta_{q} w(s) - \int_0^\cdot f(s) \dd (\Delta_{q} v)(s) \Delta_{p} w(s)\Bigr) \Biggr]
   \end{align}
   converges in $\C^{\alpha+\beta+\gamma-\varepsilon}$ for all $\varepsilon > 0$. Moreover,
   \begin{align*}
      \lVert C(f,v,w)\rVert_{\alpha + \beta + \gamma} \lesssim \lVert f\rVert_\alpha \lVert v \rVert_\beta \lVert w \rVert_\gamma.
   \end{align*}
\end{prop}

\begin{proof}
   We only argue for the first difference in~\eqref{e:commutator 1 def}, i.e. for
   \begin{align}\label{e:commutator 1 pr1}
      X_N := \sum_{p\le N} \sum_{q<p} \left[ \int_0^\cdot \Delta_p (\pi_<(f,v))(s) \dd \Delta_{q} w(s) - \int_0^\cdot f(s) \Delta_p v(s) \dd \Delta_{q} w(s) \right].
   \end{align}
   The second difference can be handled using the same arguments. First we prove that $(X_N)$ converges uniformly, then we show that $\lVert X_N \rVert_{\alpha + \beta + \gamma}$ stays uniformly bounded. This will imply the desired result, since bounded sets in $\C^{\alpha+\beta+\gamma}$ are relatively compact in $\C^{\alpha+\beta+\gamma-\varepsilon}$.

   To prove uniform convergence, note that
   \begin{align}\label{e:commutator 1 pr2}\nonumber
      X_N - X_{N-1} & = \sum_{q<N}\left[ \int_0^\cdot \Delta_N (\pi_<(f,v))(s) \dd \Delta_{q} w(s) - \int_0^\cdot f(s) \Delta_N v(s) \dd \Delta_{q} w(s) \right]\\ \nonumber
      & = \sum_{q<N} \Biggl[ \sum_{j\le N} \sum_{i<j} \int_0^\cdot \Delta_N (\Delta_i f \Delta_j v)(s) \dd \Delta_{q} w(s)\\
      &\hspace{40pt} - \sum_{j \ge N} \sum_{i\le j} \int_0^\cdot \Delta_j (\Delta_i f \Delta_N v)(s) \dd \Delta_{q} w(s) \Biggr],
   \end{align}
   where for the second term it is possible to take the infinite sum over $j$ outside of the integral because $\sum_j \Delta_j g$ converges uniformly to $g$ and because $\Delta_q w$ is a finite variation path. We also used that $\Delta_N (\Delta_i f \Delta_j v)=0$ whenever $i > N$ or $j > N$. Only very few terms in \eqref{e:commutator 1 pr2} cancel. Nonetheless these cancellations are crucial, since they eliminate most terms for which we only have the worse estimate \eqref{e:schauder blocks product bad} in Corollary~\ref{c:schauder blocks product}. We obtain
   \begin{align}\label{e:commutator 1 pr3}\nonumber
      X_N - X_{N-1}& = \sum_{q<N} \sum_{j<N} \sum_{i<j} \int_0^\cdot \Delta_N (\Delta_i f \Delta_j v)(s) \dd \Delta_{q} w(s) -  \sum_{q<N} \int_0^\cdot \Delta_N (\Delta_N f \Delta_N v)(s) \dd \Delta_{q} w(s) \\ \nonumber
      &\hspace{20pt} - \sum_{q<N} \sum_{j > N} \sum_{i<j} \int_0^\cdot \Delta_j (\Delta_i f \Delta_N v)(s) \dd \Delta_{q} w(s) \\
      &\hspace{20pt} - \sum_{q<N} \sum_{j > N} \int_0^\cdot \Delta_j (\Delta_j f \Delta_N v)(s) \dd \Delta_{q} w(s).
   \end{align}
   Note that $\lVert \partial_t \Delta_q w\rVert_\infty \lesssim 2^q \lVert \Delta_q w \rVert_\infty$. Hence, an application of Corollary~\ref{c:schauder blocks product}, where we use \eqref{e:schauder blocks product good} for the first three terms and \eqref{e:schauder blocks product bad} for the fourth term, yields
   \begin{align}\label{e:commutator 1 pr convergence speed} \nonumber
      \lVert X_N - X_{N-1} \rVert_\infty &\lesssim  \lVert f \rVert_\alpha \lVert v \rVert_\beta \lVert w\rVert_\gamma \Biggl[ \sum_{q<N} \sum_{j<N} \sum_{i<j} 2^{-2N + i + j} 2^{-i\alpha} 2^{-j\beta}2^{q(1-\gamma)} \\ \nonumber
      &\qquad + \sum_{q<N} 2^{-N(\alpha + \beta)} 2^{q(1-\gamma)} + \sum_{q<N} \sum_{j > N} \sum_{i<j} 2^{-2j+i+N} 2^{-i\alpha} 2^{-N\beta} 2^{q(1-\gamma)} \\ \nonumber
      &\qquad +  \sum_{q<N} \sum_{j > N} 2^{-j\alpha} 2^{-N\beta} 2^{q(1-\gamma)}\Biggr] \\
      &\lesssim \lVert f \rVert_\alpha \lVert v \rVert_\beta \lVert w\rVert_\gamma 2^{-N(\alpha + \beta + \gamma - 1)},
   \end{align}
   where in the last step we used $\alpha, \beta, \gamma < 1$. Since $\alpha + \beta + \gamma > 1$, this gives us the uniform convergence of $(X_N)$.

   Next let us show that $\lVert X_N \rVert_{\alpha + \beta + \gamma} \lesssim \lVert f \rVert_\alpha \lVert v \rVert_\beta \lVert w\rVert_\gamma$ for all $N$. Similarly to \eqref{e:commutator 1 pr3} we obtain for $n \in \N$
   \begin{align*}
      \Delta_n X_N &= \sum_{p\le N} \sum_{q<p}\Delta_n \Biggl[ \sum_{j<p} \sum_{i<j} \int_0^\cdot \Delta_p (\Delta_i f \Delta_j v)(s) \dd \Delta_{q} w(s) - \int_0^\cdot \Delta_p (\Delta_p f \Delta_p v)(s) \dd \Delta_{q} w(s)\\
      &\hspace{80pt}- \sum_{j>p} \sum_{i\le j} \int_0^\cdot \Delta_j (\Delta_i f \Delta_p v)(s) \dd \Delta_{q} w(s) \Biggr],
   \end{align*}
   and therefore by Corollary~\ref{c:schauder blocks}
   \begin{align*}
      \lVert \Delta_n X_N\rVert_\infty &\lesssim \sum_{p} \sum_{q<p} \Biggl[ \sum_{j<p} \sum_{i<j} 2^{-(n\vee p) - n + p + q} \lVert \Delta_p (\Delta_i f \Delta_j v)\rVert_\infty \lVert \Delta_{q} w\rVert_\infty \\
      &\hspace{60pt} + 2^{-(n\vee p) - n + p + q}\lVert \Delta_p (\Delta_p f \Delta_p v)\rVert_\infty \lVert \Delta_{q} w\rVert_\infty\\
      &\hspace{60pt} + \sum_{j>p} \sum_{i\le j} 2^{-(n\vee j) - n + j + q} \lVert\Delta_j(\Delta_i f \Delta_p v)\rVert_\infty \lVert \Delta_{q} w\rVert_\infty \Biggr].
   \end{align*}
   Now we apply Corollary~\ref{c:schauder blocks product}, where for the last term we distinguish the cases $i < j$ and $i = j$. Using that $1-\gamma > 0$, we get
   \begin{align*}
      \lVert \Delta_n X_N\rVert_\infty & \lesssim  \lVert f \rVert_\alpha \lVert v \rVert_\beta \lVert w\rVert_\gamma \sum_p 2^{p(1-\gamma)} \Biggl[ \sum_{j<p} \sum_{i<j} 2^{-(n\vee p) - n + p} 2^{-2p} 2^{i(1-\alpha)} 2^{j(1-\beta)} \\
      &\hspace{150pt} + 2^{-(n\vee p) - n + p} 2^{-p\alpha} 2^{-p\beta}\\
      &\hspace{150pt} + \sum_{j>p} \sum_{i < j} 2^{-(n\vee j) - n + j} 2^{-2j + i(1-\alpha) + p (1-\beta)}\\
      &\hspace{150pt} + \sum_{j>p} 2^{-(n\vee j) - n + j} 2^{-j\alpha - p \beta} \Biggr]\\
       &\lesssim  \lVert f \rVert_\alpha \lVert v \rVert_\beta \lVert w\rVert_\gamma 2^{-n(\alpha+\beta+\gamma)},
   \end{align*}
   where we used both that $\alpha+\beta+\gamma>1$ and that $\beta+\gamma<1$.
\end{proof}

\begin{rmk}
   If $\beta + \gamma = 1$, we can apply Proposition~\ref{p:commutator 1} with $\beta - \varepsilon$ to obtain that $C(f,v,w) \in \C^{\alpha + \beta + \gamma - \varepsilon}$ for every sufficiently small $\varepsilon > 0$. If $\beta + \gamma > 1$, then we are in the Young setting and there is no need to introduce the commutator.
\end{rmk}

For later reference, we collect the following result from the proof of Proposition~\ref{p:commutator 1}:

\begin{lem}\label{l:commutator speed of convergence}
   Let $\alpha, \beta, \gamma, f, v, w$ be as in Proposition~\ref{p:commutator 1}. Then
   \[
      \lVert C(f,v,w) - L(S_N(\pi_<(f, v)), S_N w) - I(f,\dd L(S_N v,S_N w))\rVert_\infty \lesssim 2^{-N(\alpha + \beta + \gamma - 1)} \lVert f \rVert_\alpha \lVert v \rVert_\beta \lVert w\rVert_\gamma.
   \]
\end{lem}

\begin{proof}
   Simply sum up~\eqref{e:commutator 1 pr convergence speed} over $N$.
\end{proof}

\subsection{Pathwise integration for paracontrolled paths}\label{s:schauder rough path}

In this section we apply the commutator estimate to construct the rough path integral under the assumption that the L\'evy area exists for a given reference path.%show that if $v,w \in \C^\alpha$, and if $I(v,\dd w) \in \C^{\alpha}$ is given and satisfies $I(v,\dd w) - \pi_<(v,w) \in \C^{2\alpha}$, then we can construct the pathwise integral $I(f,\dd v)$ for all $f \in \D^\beta_v$.% and $g \in \D^\alpha_w$.

\begin{thm}\label{t:rough path integral}
   Let $\alpha \in (1/3,1)$, $\beta \in (0,\alpha]$ and assume that $2\alpha+\beta>0$ as well as $\alpha+\beta \neq 1$. Let $v \in \C^\alpha(\R^d)$ and assume that the L\'evy area
   \begin{align*}
      L(v,v) := \lim_{N \rightarrow \infty}\bigl( L(S_N v^k, S_N v^\ell) \bigr)_{1\le k \le d, 1\le \ell \le d}
   \end{align*}
   converges uniformly and that $\sup_N \lVert L(S_N v, S_N v) \rVert_{2\alpha} < \infty$. Let $f \in \D^\alpha_v(\L(\R^d, \R^m))$. Then $I(S_N f, \dd S_N v)$ converges in $\C^{\alpha - \varepsilon}$ for all $\varepsilon > 0$. Denoting the limit by $I(f,\dd v)$, we have
   \begin{align*}
      \lVert I(f,\dd v)\rVert_\alpha \lesssim \lVert f \rVert_{v,\beta} \bigl(\lVert v \rVert_\alpha + \lVert v \rVert_\alpha^2 + \lVert L(v,v) \rVert_{2\alpha}\bigr).
   \end{align*}
   Moreover, $I(f,\dd v) \in \D^{\alpha}_v$ with derivative $f$ and
   \begin{align*}
      \lVert I(f,\dd v) \rVert_{v,\alpha} \lesssim \lVert f \rVert_{v,\beta} \bigl(1 + \lVert v \rVert_\alpha^2 + \lVert L(v,v) \rVert_{2\alpha}\bigr).
   \end{align*}
\end{thm}

\begin{proof}
   If $\beta+\gamma >1$, everything follows from the Young case, Theorem~\ref{t:young integral}, so let $\beta+\gamma<1$. We decompose
   \begin{align*}
      I(S_N f, \dd S_N v) & = S(S_N f, S_N v) + \pi_<(S_N f, S_N v) + L(S_N f^\sharp, S_N v) \\
      &\quad + [L(S_N \pi_<(f^v,v), S_N v) - I(f^v, \dd L(S_N v,S_N v))] + I(f^v, \dd L(S_N v,S_N v)).
   \end{align*}
   Convergence then follows from Proposition~\ref{p:commutator 1} and Theorem~\ref{t:young integral}. The limit is given by
   \[
      I(f,\dd v) = S(f,v) + \pi_<(f,v) + L(f^\sharp, v) + C(f^v,v,v) + I(f^v, \dd L(v,v)),
   \]
   from where we easily deduce the claimed bounds.
\end{proof}

\begin{rmk}\label{r:locality of integral}
   Since $I(f,\dd v) = \lim_{N\to \infty} \int_0^\cdot S_N f \dd S_N v$, the integral is a local operator in the sense that $I(f,\dd v)$ is constant on every interval $[s,t]$ for which $f|_{[s,t]}=0$. In particular we can estimate $I(f,\dd v)|_{[0,t]}$ using only $f|_{[0,t]}$ and $f^v|_{[0,t]}$.
\end{rmk}

For fixed $v$ and $L(v,v)$, the map $f \mapsto I(f,\dd v)$ is linear and bounded from $\D^\beta_v$ to $\D^\alpha_v$, and this is what we will need to solve differential equations driven by $v$. But we can also estimate the speed of convergence of $I(S_N f, \dd S_N v)$ to $I(f, \dd v)$, measured in uniform distance rather than in $\C^\alpha$:

\begin{cor}\label{c:rough path speed of convergence}
   Let $\alpha \in (1/3,1/2]$ and let $\beta,v,f$ be as in Theorem~\ref{t:rough path integral}. Then we have for all $\varepsilon \in (0, 2\alpha + \beta-1)$
   \begin{align*}
      \lVert I(S_N f, \dd S_N v) - I(f,\dd v)\rVert_\infty &\lesssim_\varepsilon 2^{-N(2\alpha + \beta - 1 - \varepsilon)} \lVert f\rVert_{v,\beta} \bigl(\lVert v \rVert_\alpha + \lVert v \rVert_\alpha^2 \bigr)\\
      &\qquad +\lVert f^v \rVert_\beta \lVert L(S_N v, S_N v) - L(v,w)\rVert_{2\alpha-\varepsilon}.
   \end{align*}
\end{cor}

\begin{proof}
   We decompose $I(S_N f, \dd S_N v)$ as described in the proof of Theorem~\ref{t:rough path integral}. This gives us for example the term
   \[
      \| \pi_<(S_N f - f, S_N v) + \pi_<(f, S_N v - v)\|_\infty \lesssim_\varepsilon \| S_N f - f\|_\infty \| v \|_\alpha + \| f \|_\infty \| f \|_\alpha \|S_N v - v\|_\varepsilon
   \]
   for all $\varepsilon > 0$. From here it is easy to see that
   \[
      \| \pi_<(S_N f - f, S_N g) + \pi_<(f, S_N g - g)\|_\infty \lesssim 2^{-N(\alpha-\varepsilon)} \|f\|_\alpha \|v\|_\alpha \lesssim 2^{-N(\alpha-\varepsilon)} \|f\|_{v,\beta} (\|v\|_\alpha + \|v\|_\alpha^2).
   \]
   But now $\beta\le \alpha \le 1/2$ and therefore $\alpha \ge 2\alpha + \beta - 1$.

   Let us treat one of the critical terms, say $L(S_N f^\sharp, S_N v) - L(f^\sharp, v)$. Since $2 \alpha + \beta - \varepsilon > 1$, we can apply Lemma~\ref{l:Levy area regularity} to obtain
   \begin{align*}
      &\lVert L(S_N f^\sharp, S_N v) - L(f^\sharp, v) \rVert_\infty \lesssim \lVert L(S_N f^\sharp - f^\sharp, S_N v)\rVert_{1+\varepsilon} + \lVert L(f^\sharp, S_N v - v)\rVert_{1+\varepsilon}\\
      &\hspace{120pt}\lesssim_\varepsilon \lVert S_N f^\sharp - f^\sharp\rVert_{1+\varepsilon - \alpha} \lVert v \rVert_\alpha + \lVert f^\sharp\rVert_{\alpha+\beta} \lVert S_N v - v\rVert_{1+\varepsilon - \alpha-\beta} \\
      &\hspace{120pt}\lesssim 2^{-N(\alpha + \beta - (1 + \varepsilon - \alpha))} \lVert f^\sharp\rVert_{\alpha+\beta} \lVert v \rVert_\alpha + 2^{-N(\alpha - (1+\varepsilon - \alpha-\beta))} \lVert f^\sharp\rVert_{\alpha+\beta} \lVert v \rVert_\alpha \\
      &\hspace{120pt}\lesssim 2^{-N(2\alpha + \beta - 1 -\varepsilon)}\lVert f^\sharp \rVert_{\alpha+\beta} \lVert v \rVert_\alpha.
   \end{align*}
   Lemma~\ref{l:commutator speed of convergence} gives
   \begin{align*}
      \lVert L(S_N \pi_<(f^v,v), S_N v) - L(\pi_<(f^v,v),v) \rVert_\infty &\lesssim 2^{-N(2 \alpha + \beta - 1)} \lVert f^v \rVert_\beta \lVert v\rVert_\alpha^2\\
      &\quad + \lVert I(f^v, \dd L(S_N v, S_N v)) - I(f^v, \dd L(v,v))\rVert_\infty.
   \end{align*}
   The second term on the right hand side can be estimated using the continuity of the Young integral, and the proof is complete.
\end{proof}

\begin{rmk}
   In Lemma~\ref{l:commutator speed of convergence} we saw that the rate of convergence of
   \[
      L(S_N \pi_<(f^v,v),S_N v) - I(f^v, \dd L(S_Nv, S_Nv)) - (L(\pi_<(f^v,v),v) - I(f^v, \dd L(v,v)))
   \]
   is in fact $2^{-N(2\alpha+\beta - 1)}$ when measured in uniform distance, and not just $2^{-N(2\alpha +\beta- 1 -\varepsilon)}$. It is possible to show that this optimal rate is attained by the other terms as well, so that
   \begin{align*}
      \lVert I(S_N f, \dd S_N v) - I(f,\dd v)\rVert_\infty &\lesssim 2^{-N(2\alpha + \beta - 1)} \lVert f\rVert_{v,\beta} \bigl(\lVert v \rVert_\alpha + \lVert v \rVert_\alpha^2 \bigr)\\
      &\qquad +\lVert f^v \rVert_\beta \lVert L(S_N v, S_N w) - L(v,w)\rVert_{2\alpha - \varepsilon}.
   \end{align*}
   Since this requires a rather lengthy calculation, we decided not to include the arguments here.
\end{rmk}

Since we approximate $f$ and $g$ by the piecewise smooth functions $S_N f$ and $S_N g$ when defining the integral $I(f,\dd g)$, it is not surprising that we obtain a Stratonovich type integral:

\begin{prop}\label{p:ibp stratonovich}
   Let $\alpha \in (1/3,1)$ and $v \in \C^\alpha(\R^d)$. Let $\varepsilon > 0$ be such that $(2+\varepsilon)\alpha > 1$ and let $F \in C^{2+\varepsilon}(\R^d,\R)$. Then
   \begin{align*}
      F(v(t)) - F(v(0)) = I(\dD F(v),\dd v)(t) := \lim_{N\rightarrow \infty} I(S_N \dD F(v), \dd S_N v)(t)
   \end{align*}
   for all $t \in [0,1]$.
\end{prop}

\begin{proof}
   The function $S_N v$ is Lipschitz continuous, so that integration by parts gives
   \begin{align*}
      F(S_N v(t)) - F(S_N v(0)) = I(\dD F(S_N v), \dd S_N v)(t).
   \end{align*}
   The left hand side converges to $F(v(t)) - F(v(0))$. It thus suffices to show that $I(S_N \dD F(v)-\dD F(S_N v), \dd S_N v)$ converges to zero. By continuity of the Young integral, Theorem~\ref{t:young integral}, it suffices to show that $\lim_{N\rightarrow \infty} \lVert S_N \dD F(v) - \dD F(S_N v)\rVert_{\alpha(1+\varepsilon')} = 0$ for all $\varepsilon' < \varepsilon$. Recall that $S_N v$ is the linear interpolation of $v$ between the points $(t^1_{pm})$ for $p \le N$ and $0 \le m \le 2^p$, and therefore $\Delta_p \dD F(S_Nv) = \Delta_p \dD F(v) = \Delta_p S_N \dD F(v)$ for all $p \le N$. For $p > N$ and $1 \le m \le 2^p$ we apply a first order Taylor expansion to both terms and use the $\varepsilon$--H\"older continuity of $\dD^2 F$ to obtain
   \begin{align*}
      \left|[S_N \dD F(v) - \dD F(S_N v)]_{pm}\right| & \le C_F 2^{-p\alpha(1+\varepsilon)} \lVert S_N v \rVert_\alpha
   \end{align*}
   for a constant $C_F>0$. Therefore, we get for all $\varepsilon' \le \varepsilon$
   \begin{align*}
      \lVert S_N \dD F(v) - \dD F(S_Nv)\rVert_{\alpha(1+ \varepsilon')} \lesssim_F 2^{-N\alpha(\varepsilon-\varepsilon')} \lVert v \rVert_\alpha,
   \end{align*}
   which completes the proof.
\end{proof}

\begin{rmk}\label{r:symmetric structure induces cancellations stratonovich}
   Note that here we did not need any assumption on the area $L(v,v)$. The reason are cancellations that arise due to the symmetric structure of the derivative of $\dD F$, the Hessian of $F$. %It is possible to show directly for the L\'evy area $L(S_N \dD F(v), S_N v)$ that it converges, and in this approach the importance of the symmetry of $\dD F$ becomes very obvious: After a first order Taylor expansion, the fact that $L$ is antisymmetric implies, in conjunction with the symmetry of the Hessian of $F$, that all terms cancel except the remainder, which is smooth enough to be accessible with Young integration. However, the disadvantage of this approach is that then it is not trivial to identify the limit as $F(v(t)) - F(v(0))$. That is why we chose the approach presented above.

   Proposition~\ref{p:ibp stratonovich} was previously obtained by Roynette~\cite{Roynette1993}, except that there $v$ is assumed to be one dimensional and in the Besov space $B^{1/2}_{1,\infty}$.
\end{rmk}

\section{Pathwise It\^{o} integration}\label{s:pathwise ito}

In the previous section we saw that our pathwise integral $I(f,\dd v)$ is of Stratonovich type, i.e. it satisfies the usual integration by parts rule. But in applications it may be interesting to have an It\^{o} integral. Here we show that a slight modification of $I(f,\dd v)$ allows us to treat non-anticipating It\^{o}-type integrals.

A natural approximation of a non-anticipating integral is given for $k \in \N$ by
\begin{align*}
   \ito_k (f,\dd v) (t) :=\, & \sum_{m=1}^{2^k} f(t^0_{km}) (v(t^2_{km}\wedge t) - v(t^0_{km}\wedge t))\\
   =\, & \sum_{m=1}^{2^k} \sum_{p,q} \sum_{m,n} f_{pm} v_{qn} \varphi_{pm}(t^0_{km}) (\varphi_{qn}(t^2_{km}\wedge t) - \varphi_{qn}(t^0_{km}\wedge t)).
\end{align*}
Let us assume for the moment that $t=m2^{-k}$ for some $0 \le m \le 2^k$. In that case we obtain for $p \ge k$ or $q \ge k$ that $\varphi_{pm}(t^0_{km})(\varphi_{qn}(t^2_{km}\wedge t) - \varphi_{qn}(t^0_{km}\wedge t)) = 0$. For $p,q<k$, both $\varphi_{pm}$ and $\varphi_{qn}$ are affine functions on $[t^0_{km}\wedge t, t^2_{k m}\wedge t]$, and for affine $u$ and $w$ and $s<t$ we have
\begin{align*}
   u(s)(w(t) - w(s)) = \int_s^t u(r) \dd w(r) - \frac{1}{2} [u(t) - u(s)] [w(t) - w(s)].
\end{align*}
Hence, we conclude that for $t=m2^{-k}$
\begin{align}\label{e:ito via piecewise linear}
   \ito_k (f,\dd v)(t) = I(S_{k-1} f, \dd S_{k-1} v)(t) - \frac{1}{2}[f,v]_k(t),
\end{align}
where $[f,v]_k$ is the $k$--th dyadic approximation of the quadratic covariation $[f,v]$, i.e.
\begin{align*}
   [f,v]_k(t) := \sum_{m=1}^{2^k} [f(t^2_{km}\wedge t) - f(t^0_{km}\wedge t)][v(t^2_{km}\wedge t) - v(t^0_{km}\wedge t)].
\end{align*}
From now on we study the right hand side of~\eqref{e:ito via piecewise linear} rather than $\ito_k(f,\dd v)$, which is justified by the following remark.

\begin{rmk}\label{r:our ito vs nonanticipating riemann}
   Let $\alpha \in (0,1)$. If $f\in C([0,1])$ and $v\in \C^\alpha$, then
   \begin{align*}
      \Bigl\lVert \ito_k(f,\dd v) - \Bigl(I(S_{k-1} f,\dd S_{k-1}v) - \frac{1}{2} [S_{k-1} f, S_{k-1}v]_k \Bigr) \Bigr\rVert_\infty \lesssim 2^{-k\alpha} \lVert f \rVert_\infty \lVert v \rVert_\alpha.
   \end{align*}
   This holds because both functions agree in all dyadic points of the form $m2^{-k}$, and because between those points the integrals can pick up mass of at most $\lVert f \rVert_\infty 2^{-k\alpha} \lVert v \rVert_\alpha$.
\end{rmk}

%Recall that we defined $\ito_k(f,\dd g)(t) = \sum_\ell f(t^0_{k\ell}) g_{t^0_{k\ell}\wedge t, t^2_{k\ell}\wedge t}$.
%
%
%
%
%
%
%For the moment let us continue by studying the right-hand side of \eqref{e:ito via piecewise linear}. Later we will show how to return from there to $\ito_k(f,\dd v)(t)$ for general $t$, not necessarily of the form $t=m2^{-k}$.

We write $[v,v] := ([v^i, v^j])_{1 \le i, j \le d}$ and $L(v,v) := (L(v^i, v^j))_{1\le i,j\le d}$, and similarly for all expressions of the same type.

\begin{thm}\label{t:pathwise ito integral}
   Let $\alpha \in (0,1/2)$ and let $\beta\le \alpha$ be such that $2\alpha + \beta > 1$. Let $v\in \C^\alpha(\R^d)$ and $f \in \D^\beta_v(\L(\R^d;\R^n))$. Assume that $(L(S_k v, S_k v))$ converges uniformly, with uniformly bounded $\C^{2\alpha}$ norm. Also assume that $([v,v]_k)$ converges uniformly. Then $(\ito_k(f,\dd v))$ converges uniformly to a limit $\ito(f,\dd v) = I(f,\dd v) - 1/2[f,v]$ which satisfies
   \begin{align*}
      \lVert \ito(f,\dd v)\rVert_\infty \lesssim \lVert f\rVert_{v,\beta} (\lVert v \rVert_\alpha + \lVert v \rVert_\alpha^2 + \lVert L(v,v) \rVert_{2\alpha} + \lVert[v,v]\rVert_\infty), %+ \|f\|_\infty^2 \lVert[v,v]\rVert_\infty,
   \end{align*}
   and where the quadrativ variation $[f,v]$ is given by
   \begin{equation}\label{e:quadratic variation controlled explicit}
      [f,v] = \int_0^\cdot f^{v}(s) \dd [v,v](s) := \bigg( \sum_{j,\ell=1}^d\int_0^\cdot (f^{ij})^{v,\ell}(s) \dd [v^j,v^\ell](s)\bigg)_{1\le i \le n},
   \end{equation}
   where $(f^{ij})^{v,\ell}$ is the $\ell$--th component of the $v$--derivative of $f^{ij}$.
%   The quadratic variation of $\ito(f,\dd g)$ is given by
%   \begin{align}\label{e:quadratic variation controlled explicit}
%      [f,g] = \sum_{i,j=1}^d \int_0^\cdot f^{w,i}(s) g^{w,j}(s) \dd [w^i,w^j](s).
%   \end{align}
   For $\varepsilon \in (0,3\alpha-1)$ the speed of convergence can be estimated by
   \begin{align*}
      \big\lVert \ito(f,\dd v) - \ito_k(f,\dd v) \big\rVert_\infty & \lesssim_\varepsilon 2^{-k(2\alpha + \beta - 1 - \varepsilon)} \lVert f\rVert_{v,\beta} \bigl( \lVert v \rVert_\alpha + \lVert v \rVert_\alpha^2 \bigr)\\
      &\quad +\lVert f^v \rVert_\beta \lVert L(S_{k-1} v, S_{k-1} v) - L(v,v)\rVert_{2\alpha} \\
      &\quad + \lVert f^v \rVert_\infty \lVert [v,v]_k - [v,v]\rVert_{\infty}.
   \end{align*}
\end{thm}

\begin{proof}
   By Remark~\ref{r:our ito vs nonanticipating riemann}, it suffices to show our claims for $I(S_{k-1} f, \dd S_{k-1} v) -1/2[f,v]_k$. The statements for the integral $I(S_{k-1} f, \dd S_{k-1} g)$ follow from Theorem~\ref{t:rough path integral} and Corollary~\ref{c:rough path speed of convergence}. So let us us concentrate on the quadratic variation $[f,v]_k$. Recall from Example~\ref{ex:controlled old vs new} that $f \in \D^\beta_v$ if and only if $R^f_{s,t} = f_{s,t} - f^v(s) w_{s,t}$ satisfies $|R^f_{s,t}| \lesssim |t-s|^{\alpha+\beta}$. Hence
   \begin{align*}
      [f,v]^i_k (t) & = \sum_m \big(f_{t^0_{km} \wedge t, t^2_{km} \wedge t} v_{t^0_{km} \wedge t, t^2_{km} \wedge t}\big)^i\\
      & = \sum_m \big(R^f_{t^0_{km} \wedge t, t^2_{km} \wedge t} v_{t^0_{km} \wedge t, t^2_{km} \wedge t}\big)^i + \sum_{j,\ell=1}^d \sum_m (f^{ij})^{v,\ell}(t^0_{km} \wedge t) v^\ell_{t^0_{km} \wedge t, t^2_{km} \wedge t} v^j_{t^0_{km} \wedge t, t^2_{km} \wedge t}.
   \end{align*}
   It is easy to see that the first term on the right hand side is bounded by
   \[
      \Big| \sum_m \big(R^f_{t^0_{km} \wedge t, t^2_{km} \wedge t} v_{t^0_{km} \wedge t, t^2_{km} \wedge t}\big)^i \Big| \lesssim 2^{-k(2\alpha+\beta-1)} \lVert f \rVert_{v,\beta}(\lVert v \rVert_\alpha + \lVert v \rVert_\alpha^2).
   \]
   For the second term, let us fix $\ell$ and $j$. Then the sum over $m$ is just the integral of $(f^{ij})^{v,\ell}$ with respect to the signed measure $\mu^k_t = \sum_{m} \delta_{t^0_{km}} v^j_{t^0_{km} \wedge t, t^2_{km} \wedge t} v^\ell_{t^0_{km} \wedge t, t^2_{km} \wedge t}$. Decomposing $\mu^k_t$ into a positive and negative part as
   \begin{align*}
      \mu^k_t & = \frac{1}{4} \Big[\sum_m  \delta_{t^0_{km}} [(v^j+v^\ell)_{t^0_{km} \wedge t, t^2_{km} \wedge t}]^2 -\sum_m  \delta_{t^0_{km}} [(v^j - v^\ell)_{t^0_{km} \wedge t, t^2_{km} \wedge t}]^2\Big]% =: \mu^{k,+}_t - \mu^{k,-}_t,
   \end{align*}
   and similarly for $\dd \mu_t = \dd [v^j,v^\ell]_t$
   we can estimate
   \begin{align*}
      &\Big| \int_0^1 (f^{ij})^{v,\ell}(s) \mu^k_t(\dd s) - \int_0^1 (f^{ij})^{v,\ell}(s) \mu_t(\dd s) \Big| \\
      &\hspace{100pt} \lesssim \left\lVert f^v \right\rVert_\infty \left(\left\lVert [v^i+v^j]_k - [v^i + v^j]\right\rVert_\infty + \left\lVert [v^i-v^j]_k - [v^i - v^j]\right\rVert_\infty\right)\\
      &\hspace{100pt} \lesssim \left\lVert f^v \right\rVert_\infty \lVert [v,v]_k - [v,v]\rVert_\infty,
   \end{align*}
   where we write $[u] := [u,u]$ and similarly for $[u]_k$. By assumption the right hand side converges to zero, from where we get the uniform convergence of $[f,g]_k$ to $[f,g]$.
%
%    Moreover, we have the explicit representation
%   \begin{align*}
%      [f,g](t) = \sum_{i,j} \int_0^t f^{w,i}(s) g^{w,j}(s) \dd [w^i,w^j](s),
%   \end{align*}
%   and therefore $\lVert [f,g]\rVert_\infty \lesssim \lVert f^w\rVert_\infty \lVert g^w\rVert_\infty \lVert [w,w]\rVert_\infty$, where we use the decomposition of $[w^i, w^j]$ into the difference of two nondecreasing processes, $[w^i,w^j] = 1/4 ([w^i + w^j] - [w^i - w^j])$.
%
%   Now let us come to the integral $I(S_{k-1} f, \dd S_{k-1} g)$. Here we only need to apply Theorem~\ref{t:rough path integral} to obtain convergence to $I(f,\dd g)$ with
%   \begin{align*}
%      \lVert I(f,\dd g)\rVert_\infty \lesssim \lVert f\rVert_{w,\alpha} \lVert g\rVert_{w,\alpha} (1 + \lVert w \rVert_\alpha^2 + \lVert L(w,w) \rVert_{2\alpha}),
%   \end{align*}
%   where we used that $1 + \lVert w \rVert_\alpha + \lVert w \rVert_\alpha^2 \lesssim 1 + \lVert w \rVert_\alpha^2$.
%   According to Corollary~\ref{c:rough path speed of convergence}, the speed of convergence can be estimated by
%   \begin{align*}
%      &\left\lVert I(f,\dd g) - I(S_{k-1} f, \dd S_{k-1} g)\right\rVert_\infty \lesssim_\varepsilon 2^{-k(3\alpha - 1 - \varepsilon)} \lVert f\rVert_{w,\alpha} \lVert g \rVert_{w,\alpha}\bigl( 1 + \lVert w \rVert_\alpha + \lVert w \rVert_\alpha^2 \bigr)\\
%      &\hspace{180pt} +\lVert f \rVert_\alpha \lVert g \rVert_\alpha \lVert L(S_{k-1} w, S_{k-1} w) - L(w,w)\rVert_{2\alpha}.
%   \end{align*}
\end{proof}

\begin{rmk}
   We calculate the pathwise It\^o integral $\ito(f,\dd v)$ as limit of nonanticipating Riemann sums involving only $f$ and $v$. This is interesting for applications of mathematical finance, because the integral process has a natural interpretation as capital obtained from investing. The classical rough path integral, see Proposition~\ref{p:Gubinelli rough paths}, is obtained via ``compensated Riemann sums'' that explicitly depend on $f^v$ and $\ito(v,\dd v)$.
\end{rmk}

\begin{rmk}
   We calculate the pathwise It\^{o} integral $\ito(f,\dd v)$ as limit of nonanticipating Riemann sums involving only $f$ and $v$. %, and not the L\'evy area of $L(v,v)$ or the quadratic variation $[v,v]$.
   The classical rough path integral, see Proposition~\ref{p:Gubinelli rough paths}, is obtained via ``compensated Riemann sums'' that depend explicitly on the derivative $f^v$ and the iterated integrals of $v$. For applications in mathematical finance, it is more convenient to have an integral that is the limit of nonanticipating Riemann sums, because this can be interpreted as capital process obtained from investing.
\end{rmk}

Note that $[v,v]$ is always a continuous function of bounded variation, but a priori it is not clear whether it is in $\C^{2\alpha}$. Under this additional assumption we have the following stronger result.

\begin{cor}\label{c:pathwise ito with smooth quadratic variation}
   In addition to the conditions of Theorem~\ref{t:pathwise ito integral}, assume that also $[v,v] \in \C^{2\alpha}$. Then $\ito(f,\dd v) \in \D^\alpha_v$ with derivative $f$, and
   \begin{align*}
      \lVert \ito(f,\dd v) \rVert_{v,\alpha} \lesssim \lVert f \rVert_{v,\beta} \bigl(1 + \lVert v \rVert_\alpha^2+ \lVert L(v,v) \rVert_{2\alpha} + \lVert [v,v]\rVert_{2\alpha} \bigr).
   \end{align*}
%   Let moreover $\tilde{w} \in \C^\alpha(\R^d)$ with L\'evy area $L(S_k\tilde{w}, S_k \tilde{w})$ that converges uniformly and with uniformly bounded $\C^{2\alpha}$ norm to $L(\tilde{w}, \tilde{w})$, and with quadratic variation $[\tilde{w},\tilde{w}]_k$ that converges uniformly to $[\tilde{w},\tilde{w}] \in \C^{2\alpha}$. Let $\tilde{f},\tilde{g} \in \D^\alpha_{\tilde{w}}(\R)$. Then
%    \begin{align*}
%       \lVert \ito(f,\dd g) - \ito(\tilde{f}, d\tilde{g})\rVert_\alpha & \lesssim \bigl( \lVert f - \tilde{f}\rVert_\alpha + \lVert f^w - \tilde{f}^{\tilde{w}}\rVert_\alpha + \lVert f^\sharp - \tilde{f}^\sharp\rVert_{2\alpha}\bigr) \lVert g \rVert_{w,\alpha}\\
%       &\hspace{70pt}  \times \bigl(1 + \lVert w \rVert_\alpha^2 + \lVert L(w,w) \rVert_{2\alpha} + \lVert [w,w]\rVert_{2\alpha}\bigr) \\
%       & + \bigl( \lVert g - \tilde{g}\rVert_\alpha + \lVert g^w - \tilde{g}^{\tilde{w}}\rVert_\alpha + \lVert g^\sharp - \tilde{g}^\sharp\rVert_{2\alpha}\bigr) \lVert \tilde{f} \rVert_{\tilde{w},\alpha} \\
%       &\hspace{70pt} \times \bigl(1 + \lVert w \rVert_\alpha^2 + \lVert L(w,w) \rVert_{2\alpha} + \lVert [w,w]\rVert_{2\alpha}\bigr) \\
%       & + \bigl(\lVert w - \tilde{w}\rVert_\alpha + \lVert L(w,w) - L(\tilde{w},\tilde{w})\rVert_{2\alpha} + \lVert [w,w] - [\tilde{w},\tilde{w}]\rVert_{2\alpha}\bigr)  \\
%       &\hspace{70pt} \times \lVert \tilde{f}\rVert_{\tilde{w},\alpha} \lVert \tilde{g}\rVert_{\tilde{w},\alpha} \bigl(1 + \lVert \tilde{w}\rVert_\alpha + \lVert w \rVert_\alpha\bigr).
%    \end{align*}
\end{cor}

\begin{proof}
   This is a combination of Theorem~\ref{t:rough path integral} and the explicit representation \eqref{e:quadratic variation controlled explicit} together with the continuity of the Young integral, Theorem~\ref{t:young integral}.
\end{proof}

The term $I(S_{k-1}f,\dd S_{k-1}v)$ has the pleasant property that if we want to refine our calculation by passing from $k$ to $k+1$, then we only have to add the additional term $I(S_{k-1}f, \dd \Delta_k v) + I(\Delta_k f, \dd S_k v)$. For the quadratic variation $[f,v]_k$ this is not exactly true. But $[f,v]_k(m2^{-k}) = [S_{k-1}f,S_{k-1}v]_k(m2^{-k})$ for $m=0,\dots, 2^k$, and there is a recursive way of calculating $[S_{k-1}f, S_{k-1}v]_k$:

\begin{lem}
   Let $f,v \in C([0,1],\R)$. Then
   \begin{align}\label{e:recursive quadratic variation}
      [S_k f,S_k v]_{k+1}(t) & = \frac{1}{2} [S_{k-1} f, S_{k-1}v]_k(t) + [S_{k-1} f, \Delta_k v]_{k+1}(t) + [\Delta_k f, S_k v]_{k+1}(t) + R_k(t)
   \end{align}
   for all $k\ge 1$ and all $t \in [0,1]$, where
   \begin{align*}
      R_k(t) := -\frac{1}{2} f_{\llcorner t^k \lrcorner,t} v_{\llcorner t^k \lrcorner,t} + f_{\llcorner t^k \lrcorner,\ulcorner t^{k+1}\urcorner  \wedge t} v_{\llcorner t^k \lrcorner,\ulcorner t^{k+1}\urcorner  \wedge t} + f_{\ulcorner t^{k+1}\urcorner  \wedge t, t} v_{\ulcorner t^{k+1}\urcorner  \wedge t,t}
   \end{align*}
   and $\llcorner t^k \lrcorner := \lfloor t 2^k \rfloor 2^{-k}$ and $\ulcorner t^{k}\urcorner := \llcorner t^k \lrcorner + 2^{-(k+1)}$. In particular, we obtain for $t=1$ that
   \begin{align}\label{e:cesaro formula quadratic variation}
      [f,v]_{k+1}(1) = \frac{1}{2}[f,v]_k(1) + \frac{1}{2} \sum_m f_{km} v_{km} = \frac{1}{2^{k+1}}\sum_{p\le k} \sum_m 2^{p} f_{pm} v_{pm}.
   \end{align}
   If moreover $\alpha \in (0,1)$ and $f,v \in \C^\alpha$, then $\lVert [S_{k-1} f, S_{k-1} g]_k - [f,g]_k \rVert_\infty \lesssim 2^{-2k\alpha} \lVert f \rVert_\alpha \lVert g \rVert_\alpha$.
\end{lem}

\begin{proof}
   Equation~\eqref{e:recursive quadratic variation} follows from a direct calculation using the fact that $S_{k-1} f$ and $S_{k-1} v$ are affine on every interval $[t^0_{k\ell},t^1_{k\ell}]$ respectively $[t^1_{k\ell},t^2_{k\ell}]$ for $1 \le \ell \le 2^k$.
%   Clearly it suffices to show $[S_{k-1}f, S_{k-1}v]_{k+1} = 1/2 [S_{k-1}f, S_{k-1} v]_k + R_k$. Let us assume that $t = m2^{-k}$. In that case $R_k(t) = 0$, and for every $\ell\le2^k$ we obtain
%   \begin{align*}
%      ([S_{k-1}f, S_{k-1}g]_{k+1})_{t^0_{k\ell},t^2_{k\ell}} & = \left((S_{k-1} f)_{t^0_{k\ell},t^1_{k\ell}} (S_{k-1} g)_{t^0_{k\ell},t^1_{k\ell}} + (S_{k-1} f)_{t^1_{k\ell},t^2_{k\ell}} (S_{k-1} g)_{t^1_{k\ell},t^2_{k\ell}}\right)\\
%      & = \frac{1}{2} (S_{k-1} f)_{t^0_{k\ell},t^2_{k\ell}} (S_{k-1} g)_{t^0_{k\ell},t^2_{k\ell}} = \frac{1}{2} ([S_{k-1}f, S_{k-1}g]_k)_{t^0_{k\ell},t^2_{k\ell}},
%   \end{align*}
%   where we used that $S_{k-1}f$ and $S_{k-1}g$ are affine on $[t^0_{k\ell},t^2_{k\ell}]$, and that the two intervals $[t^0_{k\ell},t^1_{k\ell}]$ and $[t^1_{k\ell},t^2_{k\ell}]$ have the same length $2^{-k-1}$. The term $R_k$ is now chosen exactly so that we also obtain the right expression for $t \in [0,1]$ that is not of the form $m 2^{-k}$.
   The formula for $[f,v]_{k+1}(1)$ follows from the that $[\Delta_p f, \Delta_q v]_{k+1}(1) = 0$ unless $p=q$, and that $[\Delta_k f, \Delta_k v]_{k+1} = 1/2 \sum_m f_{km} v_{km}$.
   The estimate for $\lVert [S_{k-1} f, S_{k-1} g]_k - [f,g]_k \rVert_\infty$ holds because the two functions agree in all dyadic points $m 2^{-k}$.%, and because between two such points the quadratic variation can pick up mass of at most $2^{-2k\alpha} \lVert f \rVert_\alpha \lVert g \rVert_\alpha$.
\end{proof}

\begin{rmk}
   The Ces\`aro mean formula \eqref{e:cesaro formula quadratic variation} makes the study of existence of the quadratic variation accessible to ergodic theory. This was previously observed by Gantert~\cite{Gantert1994}. See also Gantert's thesis~\cite{Gantert1991}, Beispiel 3.29, where it is shown that ergodicity alone (of the distribution of $v$ with respect to suitable transformations on path space) is not sufficient to obtain convergence of $([v,v]_k(1))$ as $k$ tends to $\infty$.
\end{rmk}

%Recall that we defined $\ito_k(f,\dd g)(t) = \sum_\ell f(t^0_{k\ell}) g_{t^0_{k\ell}\wedge t, t^2_{k\ell}\wedge t}$.
%
%\begin{rmk}\label{r:our ito vs nonanticipating riemann}
%   Let $\alpha \in (0,1)$. If $f\in C([0,1])$ and $g\in \C^\alpha$, then
%   \begin{align*}
%      \Bigl\lVert \ito_k(f,\dd g) - \Bigl(I(S_{k-1} f,\dd S_{k-1}g) - \frac{1}{2} [S_{k-1} f, S_{k-1}g]_k \Bigr) \Bigr\rVert_\infty \lesssim 2^{-k\alpha} \lVert f \rVert_\infty \lVert g \rVert_\alpha.
%   \end{align*}
%   This holds because both functions agree in all dyadic points of the form $m2^{-k}$, and because between those points the integrals can pick up mass of at most $\lVert f \rVert_\infty 2^{-k\alpha} \lVert g \rVert_\alpha$.
%\end{rmk}

It would be more natural to assume that for the controlling path $v$ the non-anticipating Riemann sums converge, rather than assuming that $(L(S_{k}v, S_k v))_k$ and $([v,v]_k)$ converge. This is indeed sufficient, as long as a uniform H\"older estimate is satisfied by the Riemann sums. % In that case all the conditions of Theorem~\ref{t:pathwise ito integral} and of Corollary~\ref{c:pathwise ito with smooth quadratic variation} are satisfied.
We start by showing that the existence of the It\^{o} iterated integrals implies the existence of the quadratic variation.

\begin{lem}\label{l:ito implies quadratic variation}
   Let $\alpha \in (0,1/2)$ and let $v \in \C^\alpha(\R^d)$. Assume that the non-anticipating Riemann sums $(\ito_k(v,\dd v))_k$ converge uniformly to  $\ito(v,\dd v)$. Then also $([v,v]_k)_k$ converges uniformly to a limit $[v,v]$. %Moreover, for all $0 \le s < t \le 1$
%   \begin{equation}\label{e:hoelder continuity quadratic variation}
%      |[w,w]_k(t) - [w,w]_k(s)| \lesssim |\ito_k(w,\dd w)_{s,t} - w(s) w_{s,t}| + |w_{s,t}|^2.
%   \end{equation}
   If moreover
   \begin{align}\label{e:discrete hoelder} \nonumber
      &\sup_k \sup_{0 \le m < m' \le 2^k} \frac{|\ito_k(v,\dd v)(m' 2^{-k}) - \ito_k(v,\dd v)(m 2^{-k}) - v(m2^{-k}) (v(m'2^{-k}) - v(m2^{-k}))|}{|(m'-m)2^{-k}|^{2\alpha}}\\
      &\hspace{20pt} = C < \infty,
   \end{align}
   then $[v,v] \in \C^{2\alpha}$ and $\lVert [v,v] \rVert_{2\alpha} \lesssim C + \lVert v \rVert_\alpha^2$.
\end{lem}

\begin{proof}
   Let $t \in [0,1]$ and $1 \le i,j \le d$. Then
   \begin{align*}%\label{e:quadratic variation ito}\nonumber
      & v^i(t) v^j(t) - v^i(0)v^j(0) = \sum_{m = 1}^{2^k} \left[v^i(t^2_{km}\wedge t) v^j(t^2_{km}\wedge t) - v^i(t^0_{km}\wedge t) v^j(t^0_{km}\wedge t)\right] \\ %\nonumber
      &\hspace{50pt} = \sum_{m = 1}^{2^k} \left[v^i(t^0_{km}) v^j_{t^0_{km}\wedge t, t^2_{km}\wedge t} + v^j(t^0_{km}) v^i_{t^0_{km}\wedge t, t^2_{km}\wedge t} + v^i_{t^0_{km}\wedge t, t^2_{km}\wedge t} v^j_{t^0_{km}\wedge t, t^2_{km}\wedge t}\right] \\
      &\hspace{50pt} = \ito_k(v^i,\dd v^j)(t) + \ito_k(v^j,\dd v^i)(t) + [v^i,v^j]_k(t),
   \end{align*}
   which implies the convergence of $([v,v]_k)_k$ as $k$ tends to $\infty$. For $0\le s<t\le 1$ this gives%we obtain from \eqref{e:quadratic variation ito} that
   \begin{align*}
      ([v^i,v^j]_k)_{s,t} & = \bigl(v^i v^j\bigr)_{s,t} - \ito_k(v^i,\dd v^j)_{s,t} - \ito_k(v^j,\dd v^i)_{s,t} \\
      & = \left[v^i(s) v^j_{s,t} - \ito_k(v^i,\dd v^j)_{s,t}\right] + \left[v^j(s) v^i_{s,t} - \ito_k(v^j,\dd v^i)_{s,t}\right] + v^i_{s,t} v^j_{s,t},
   \end{align*}
   At this point it is easy to estimate $\lVert [v,v]\rVert_{2\alpha}$, where we work with the classical H\"older norm and not the $\C^{2\alpha}$ norm. Indeed let $0 \le s < t \le 1$. Using the continuity of $[v,v]$, we can find $k$ and $s\le s_k = m_s 2^{-k}< m_t 2^{-k} = t_k \le t$ with $|[v,v]|_{s,s_k} + |[v,v]|_{t,t_k}\le \lVert v\rVert_\alpha^2 |t-s|^{2\alpha}$. Moreover,
   \[
      |[v,v]|_{s_k,t_k} \le \Big(\sup_{\ell \ge k} \sup_{0 \le m < m' \le 2^\ell} \frac{|([v,v]_\ell)_{m2^{-\ell},m'2^{-\ell}}|}{|(m'-m)2^{-\ell}|^{2\alpha}} \Big)|t_k - s_k|^{2\alpha}  \le (2C + \lVert v \rVert_{\alpha}^2) |t-s|^{2\alpha}.
   \]
%
%
%
%    such that
%
%
%
%   leading to \eqref{e:hoelder continuity quadratic variation}. Given \eqref{e:hoelder continuity quadratic variation} it is now easy to estimate $\lVert [v,v]\rVert_{2\alpha}$. We estimate the classical H\"older norm, not the $\C^{2\alpha}$ norm. Let $0 \le s < t \le 1$. Using the continuity of $[v,v]$, we choose $k$ large enough such that there exist $s<s_k = m_s 2^{-k}<t$ and $s<t_k = m_t 2^{-k}<t$ with
%   \begin{align*}
%      |[v,v]_{s,s_k}| + |[v,v]_{t_k,t}| + \lVert [v,v]_k - [v,v]\rVert_\infty \le \lVert v\rVert_\alpha^2 |t-s|^{2\alpha}.
%   \end{align*}
%   Since
%   \begin{align*}
%      |[v,v]_{s,t}| \le |[v,v]_{s,s_k}| + |[v,v]_{t_k,t}| + \lVert [v,v]_k - [v,v] \rVert_\infty,
%   \end{align*}
%   we obtain \eqref{e:hoelder estimate quadratic variation} as a consequence of \eqref{e:hoelder continuity quadratic variation} and the hypothesis.
\end{proof}

\begin{rmk}
   The ``coarse-grained H\"older condition''~\eqref{e:discrete hoelder} is from~\cite{Perkowski2013Pathwise} and has recently been discovered independently by~\cite{Kelly2014}.
\end{rmk}

Similarly convergence of $(\ito_k(v,\dd v))$ implies convergence of $(L(S_k v, S_k v))_k$:

\begin{lem}\label{l:ito implies stratonovich}
   In the setting of Lemma~\ref{l:ito implies quadratic variation}, assume that~\eqref{e:discrete hoelder} holds.
%   Let $\alpha \in (0,1/2)$, and let $w \in \C^\alpha(\R^d)$. Assume that the non-anticipating integrals $(\ito_k(w,\dd w))_k$ converge uniformly, and that
%   \begin{align*}
%      &\sup_k \sup_{0 \le \ell < \ell' \le 2^k} \frac{|\ito_k(w,\dd w)(\ell' 2^{-k}) - \ito_k(w,\dd w)(\ell 2^{-k}) - w(\ell2^{-k}) (w(\ell'2^{-k}) - w(\ell2^{-k}))|}{|(\ell'-\ell)2^{-k}|^{2\alpha}}\\
%      &\hspace{20pt} = C < \infty.
%   \end{align*}
   Then $L(S_k v, S_k v)$ converges uniformly as $k$ tends to $\infty$, and
   \begin{align*}
      \sup_k \lVert L(S_k v, S_k v)\rVert_{2\alpha} \lesssim C + \lVert v \rVert_\alpha^2.
   \end{align*}
\end{lem}

\begin{proof}
   Let $k \in \N$ and $0 \le m \le 2^k$, and write $t = m 2^{-k}$. Then we obtain from \eqref{e:ito via piecewise linear} that
   \begin{align}\label{e:ito implies stratonovich pr1}
      &L(S_{k-1} v, S_{k-1} v)(t) \\ \nonumber
      %&\hspace{40pt} = I(S_{k-1} w, \dd S_{k-1} w)(t) - \pi_<(S_{k-1}w, S_{k-1}w)(t) - S(S_{k-1}w, S_{k-1}w)(t)\\ \nonumber
      &\hspace{40pt} = \ito_k(v,\dd v)(t) + \frac{1}{2} [v,v]_k(t) - \pi_<(S_{k-1}v, S_{k-1}v)(t) - S(S_{k-1}v, S_{k-1}v)(t).
   \end{align}
   Let now $s,t \in [0,1]$. We first assume that there exists $m$ such that $t^0_{km} \le s < t \le t^2_{km}$. Then we use $\lVert \partial_t \Delta_q v \rVert_\infty \lesssim 2^{q(1-\alpha)} \lVert v \rVert_\alpha$ to obtain
   \begin{align}\label{e:ito implies stratonovich pr2}
      &|L(S_{k-1}v, S_{k-1}v)_{s,t}| \le \sum_{p<k}\sum_{q<p} \left| \int_s^t \Delta_p v(r) \dd \Delta_q v(r) - \int_s^t \dd \Delta_q v(r) \Delta_p v(r)\right| \\ \nonumber
      &\hspace{40pt} \lesssim \sum_{p<k} \sum_{q<p} |t-s| 2^{-p\alpha} 2^{q(1-\alpha)} \lVert v \rVert_\alpha^2 \lesssim |t-s| 2^{-k(2\alpha-1)} \lVert v \rVert_\alpha^2 \le |t-s|^{2\alpha} \lVert v \rVert_\alpha^2.
   \end{align}
   %where we used that $2\alpha - 1 < 0$, and also that $|t-s|\le 2^{-k}$ by assumption.

   Combining \eqref{e:ito implies stratonovich pr1} and \eqref{e:ito implies stratonovich pr2}, we obtain the uniform convergence of $(L(S_{k-1} v,S_{k-1} v))$ from Lemma~\ref{l:ito implies quadratic variation} and from the continuity of $\pi_<$ and $S$.

   For $s$ and $t$ that do not lie in the same dyadic interval of generation $k$, let $\ulcorner s^k\urcorner = m_s 2^{-k}$ and $\llcorner t^k\lrcorner = m_t 2^{-k}$ be such that $\ulcorner s^k\urcorner - 2^{-k} < s \le \ulcorner s^k\urcorner$ and $\llcorner t^k\lrcorner \le t < \llcorner t^k\lrcorner + 2^{-k}$. In particular, $\ulcorner s^k\urcorner\le \llcorner t^k\lrcorner$. Moreover
   \begin{align*}
      |L(S_{k-1}v, S_{k-1}v)_{s,t}| &\le |L(S_{k-1}v, S_{k-1}v)_{s,\ulcorner s^k\urcorner}| + |L(S_{k-1}v, S_{k-1}v)_{\ulcorner s^k\urcorner,\llcorner t^k\lrcorner }| \\
      &\hspace{20pt} + |L(S_{k-1}v, S_{k-1}v)_{\llcorner t^k \lrcorner,t}|.
   \end{align*}
   Using~\eqref{e:ito implies stratonovich pr2}, the first and third term on the right hand side can be estimated by $(|\ulcorner s^k\urcorner -s|^{2\alpha} + |t-\llcorner t^k \lrcorner|^{2\alpha})\lVert v \rVert_\alpha^2 \lesssim |t-s|^{2\alpha} \lVert v \rVert_\alpha^2$. For the middle term we apply \eqref{e:ito implies stratonovich pr1} to obtain
   \begin{align*}
      |L(S_{k-1}v, S_{k-1}v)_{\ulcorner s^k\urcorner,\llcorner t^k\lrcorner }| & \le \left|\ito_k(v,\dd v)_{\ulcorner s^k\urcorner,\llcorner t^k\lrcorner } - v(\ulcorner s^k\urcorner)(v(\llcorner t^k\lrcorner) - v(\ulcorner s^k\urcorner))\right| \\
      &\hspace{20pt} + \left|v(\ulcorner s^k\urcorner)v_{\ulcorner s^k\urcorner,\llcorner t^k\lrcorner } - \pi_<(S_{k-1}v,S_{k-1}v)_{\ulcorner s^k\urcorner, \llcorner t^k\lrcorner}\right| \\
      &\hspace{20pt} + \frac{1}{2} \left|([v,v]_k)_{\ulcorner s^k\urcorner,\llcorner t^k\lrcorner }\right| + \left| S(S_{k-1}v, S_{k-1}v)_{\ulcorner s^k\urcorner,\llcorner t^k\lrcorner }\right| \\
      & \lesssim |\llcorner t^k\lrcorner - \ulcorner s^k\urcorner|^{2\alpha}\left( C + \lVert v \rVert_\alpha^2\right) \le |t-s|^{2\alpha} \left(C + \lVert v \rVert_\alpha^2\right),
   \end{align*}
   where Example~\ref{ex:controlled old vs new}, Lemma~\ref{l:ito implies quadratic variation}, and Lemma~\ref{l:symmetric part} have been used.
\end{proof}

%Combining Lemma~\ref{l:ito implies quadratic variation} and Lemma~\ref{l:ito implies stratonovich} with Theorem~\ref{t:pathwise ito integral}, we see that the uniform convergence of $(\ito_k(v,\dd v))_k$ to $\ito(v,\dd v)$ implies the uniform convergence of $(\ito_k(f,\dd v))_k$ to $\ito(f,\dd v)$ for $f$ paracontrolled by $v$:
%
%\begin{cor}\label{c:ito general result}
%   Let $\alpha \in (1/3, 1/2)$, $\beta \le \alpha$ with $2\alpha+\beta>1$, and let $v\in \C^\alpha(\R^d)$ and $f \in \D^\beta_v(\R)$. Assume that $(\ito_k(v,\dd v))_k$ converges uniformly to  $\ito(v,\dd v)$, and that
%   \begin{align*}
%      &\sup_k \sup_{0 \le m < m' \le 2^k} \frac{|\ito_k(v,\dd v)(m' 2^{-k}) - \ito_k(v,\dd v)(m 2^{-k}) - v(m2^{-k}) (v(m'2^{-k}) - v(m2^{-k}))|}{|(m'-m)2^{-k}|^{2\alpha}} \\
%      &\hspace{20pt} = C < \infty.
%   \end{align*}
%   Then $(\ito_k(f,\dd v))_k$ converges uniformly to $\ito(f,\dd v) \in \D^\alpha_v$ with derivative $f$, such that
%   \begin{align*}
%      \lVert \ito(f,\dd v)\rVert_\infty \lesssim \lVert f\rVert_{v,\beta} (1 + \lVert v \rVert_\alpha^2 + C).
%   \end{align*}
%\end{cor}

It follows from the work of F\"ollmer that our pathwise It\^{o} integral satisfies It\^{o}'s formula:

\begin{cor}
   Let $\alpha \in (1/3, 1/2)$ and $v\in \C^\alpha(\R^d)$. Assume that the non-anticipating Riemann sums $(\ito_k(v,\dd v))_k$ converge uniformly to  $\ito(v,\dd v)$ %, and that furthermore
%   \begin{align*}
%      &\sup_k \sup_{0 \le m < m' \le 2^k} \frac{|\ito_k(v,\dd v)(m' 2^{-k}) - \ito_k(v,\dd v)(m 2^{-k}) - v(m2^{-k}) (v(m'2^{-k}) - v(m2^{-k}))|}{|(m'-m)2^{-k}|^{2\alpha}} \\
%      &\hspace{20pt} = C < \infty.
%   \end{align*}
   and let $F \in C^2(\R^d,\R)$. Then $(\ito_k(\dD F(v), \dd v))_k$ converges to a limit $\ito(\dD F(v), \dd v)$ that satisfies for all $t \in [0,1]$
   \begin{align*}
      F(v(t)) - F(v(0)) = \ito(\dD F(v), \dd v)(t) + \int_0^t \sum_{k,\ell=1}^d \partial_{x_k} \partial_{x_\ell} F(v(s)) \dd [v^k, v^\ell](s).
   \end{align*}
\end{cor}

\begin{proof}
   This is Remarque 1 of F\"ollmer~\cite{Follmer1979} in combination with Lemma~\ref{l:ito implies quadratic variation}.
\end{proof}

%\begin{rmk}
%   Note that $\dD F \in C^1$, and therefore $\dD F(w)$ is not paracontrolled by $w$. Just as in the Stratonovich case, see Remark~\ref{r:symmetric structure induces cancellations stratonovich}, the symmetry of the derivative of $\dD F$ leads to crucial cancellations that allow to take $\dD F$ less regular than in the non-gradient case.
%\end{rmk}

\section{Construction of the L\'evy area}\label{s:construction of levy area}

To apply our theory, it remains to construct the L\'evy area respectively the pathwise It\^{o} integrals for suitable stochastic processes. In Section~\ref{ss:hypercontractive area} we construct the L\'evy area for hypercontractive stochastic processes whose covariance function satisfies a certain ``finite variation'' property. In Section~\ref{ss:pathwise ito area for martingales} we construct the pathwise It\^{o} iterated integrals for some continuous martingales.

\subsection{Hypercontractive processes}\label{ss:hypercontractive area}

Let $X\colon [0,1] \to \R^d$ be a centered continuous stochastic process, such that $X^i$ is independent of $X^j$ for $i \neq j$.  We write $R$ for its covariance function, $R\colon [0,1]^2 \to \R^{d\times d}$ and $R(s,t) := (E(X^i_s X^j_t))_{1 \le i,j\le d}$. The increment of $R$ over a rectangle $[s,t] \times [u,v] \subseteq [0,1]^2$ is defined as
\begin{align*}
   R_{[s,t] \times [u,v]} := R(t,v) + R(s,u) - R(s,v) - R(t,u) := (E(X^i_{s,t} X^j_{u,v}))_{1 \le i, j \le d}.
\end{align*}
Let us make the following two assumptions.

\begin{itemize}
 \item[($\rho$--var)] There exists $C > 0$ such that for all $0 \le s < t \le 1$ and for every partition $s = t_0 < t_1 < \dots < t_n = t$ of $[s,t]$ we have
    \begin{align*}
       \sum_{i,j=1}^n | R_{[t_{i-1},t_i] \times [t_{j-1},t_j]}|^\rho \le C |t-s|.
    \end{align*}
 \item[(HC)] The process $X$ is hypercontractive, i.e. for every $m,n \in \N$ and every $r \ge 1$ there exists $C_{r,m,n} > 0$ such that for every polynomial $P: \R^n \rightarrow \R$ of degree $m$, for all $i_1, \dots, i_n \in \{1, \dots, d\}$, and for all $t_1, \dots, t_n \in [0,1]$
    \begin{align*}
       E(|P(X^{i_1}_{t_1}, \dots, X^{i_n}_{t_n})|^{2r}) \le C_{r,m,n} E(|P(X^{i_1}_{t_1}, \dots, X^{i_n}_{t_n})|^{2})^r.
    \end{align*}
\end{itemize}

These conditions are taken from~\cite{Friz2010c}, where under even more general assumptions it is shown that it is possible to construct the iterated integrals $I(X, \dd X)$, and that $I(X,\dd X)$ is the limit of $(I(X^n, \dd X^n))_{n \in \N}$ under a wide range of smooth approximations $(X^n)_n$ that converge to $X$.

\begin{ex}
   Condition (HC) is satisfied by all Gaussian processes. More generally, it is satisfied by every process ``living in a fixed Gaussian chaos''; see~\cite{Janson1997}, Theorem~3.50. Slightly oversimplifying things, this is the case if $X$ is given by polynomials of fixed degree and iterated integrals of fixed order with respect to a Gaussian reference process.

   Prototypical examples of processes living in a fixed chaos are Hermite processes. They are defined for $H \in (1/2,1)$ and $k\in \N$, $k \ge 1$ as
   \begin{align*}
      Z^{k,H}_t = C(H,k) \int_{\R^k} \left(\int_0^t \prod_{i=1}^k (s - y_i)^{-\left(\frac{1}{2} + \frac{1-H}{k}\right)}_+\dd s\right) \dd B_{y_1} \dots \dd B_{y_k},
   \end{align*}
   where $(B_y)_{y \in \R}$ is a standard Brownian motion, and $C(H,k)$ is a normalization constant. In particular, $Z^{k,H}$ lives in the Wiener chaos of order $k$. The covariance of $Z^{k,H}$ is
   \begin{align*}
      E(Z^{k,H}_s Z^{k,H}_t) = \frac{1}{2} \left( t^{2H} + s^{2H} + |t-s|^{2H}\right)
   \end{align*}
   Since $Z^{1,H}$ is Gaussian, it is just the fractional Brownian motion with Hurst parameter $H$. For $k=2$ we obtain the Rosenblatt process. For further details about Hermite processes see~\cite{Peccati2011}. However, we should point out that it follows from Kolmogorov's continuity criterion that $Z^{k,H}$ is $\alpha$--H\"older continuous for every $\alpha < H$. Since $H \in (1/2,1)$, Hermite processes are amenable to Young integration, and it is trivial to construct $L(Z^{k,H}, Z^{k,H})$.
\end{ex}

\begin{ex}
   Condition ($\rho$--var) is satisfied by Brownian motion with $\rho = 1$. More generally it is satisfied by the fractional Brownian motion with Hurst index $H$, for which $\rho = 1/(2H)$. It is also satisfied by the fractional Brownian bridge with Hurst index $H$. A general criterion that implies condition ($\rho$--var) is the one of Coutin and Qian~\cite{Coutin2002}: If $E(|X^i_{s,t}|^2) \lesssim |t-s|^{2H}$ and $|E(X^i_{s,s+h} X^i_{t,t+h})| \lesssim |t-s|^{2H-2} h^2$ for $i = 1, \dots, d$, then ($\rho$--var) is satisfied for $\rho = 1/(2H)$. For details and further examples see~\cite{Friz2010}, Section 15.2.
\end{ex}

\begin{lem}\label{l:rho-var to dyadic generation}
   Assume that the stochastic process $X:[0,1]\rightarrow \R$ satisfies ($\rho$--var). Then we have for all $p \ge -1$ and for all $M,N \in\N$ with $M \le N \le 2^{p}$ that
   \begin{align}\label{e:rho-var to dyadic generation}
      \sum_{m_1,m_2=M}^N |E(X_{pm_1} X_{pm_2})|^\rho \lesssim (N-M+1)2^{-p}.
   \end{align}
\end{lem}

\begin{proof}
   The case $p\le 0$ is easy so let $p \ge 1$. It suffices to note that
   \begin{align*}
      E(X_{pm_1} X_{pm_2}) & = E\left((X_{t^0_{pm_1},t^1_{pm_1}} - X_{t^1_{pm_1},t^2_{pm_1}})(X_{t^0_{pm_2},t^1_{pm_2}} - X_{t^1_{pm_2},t^2_{pm_2}})\right) \\
      & = \sum_{i_1, i_2 = 0,1} (-1)^{i_1 + i_2} R_{[t^{i_1}_{pm_1},t^{i_1+1}_{pm_1}]\times [t^{i_2}_{pm_2},t^{i_2+1}_{pm_2}]},
   \end{align*}
   and that $\{t^i_{pm}: i=0,1,2, m = M, \dots, N\}$ partitions the interval $[(M-1) 2^{-p}, N 2^{-p}]$.
%
%   The cases $p=-1$ and $p=0$ can be included by enlarging the (implicit) constant on the right hand side of \eqref{e:rho-var to dyadic generation}.
\end{proof}

\begin{lem}\label{l:generation moment estimate}
   Let $X,Y: [0,1] \rightarrow \R$ be independent, centered, continuous processes, both satisfying ($\rho$--var) for some $\rho \in [1,2]$. Then for all $i, p \ge -1$, $q<p$, and $0 \le j \le 2^i$
   \begin{align*}
      E\Big[\Big|\sum_{m\le 2^p} \sum_{n\le 2^q} X_{pm} Y_{qn} \langle 2^{-i}\chi_{ij}, \varphi_{pm} \chi_{qn}\rangle\Big|^2\Big] \lesssim 2^{(p \vee i)(1/\rho - 4)} 2^{(q \vee i)(1-1/\rho)} 2^{-i} 2^{p(4-3/\rho)} 2^{q/\rho}.
   \end{align*}
\end{lem}

\begin{proof}
   Since $p > q$, for every $m$ there exists exactly one $n(m)$, such that $\varphi_{pm}\chi_{qn(m)}$ is not identically zero. Hence, we can apply the independence of $X$ and $Y$ to obtain
   \begin{align*}
      &E\Bigl[\Bigl|\sum_{m\le 2^p} \sum_{n\le 2^q} X_{pm} Y_{qn} \langle 2^{-i}\chi_{ij}, \varphi_{pm} \chi_{qn}\rangle\Bigr|^2\Bigr] \\
      &\hspace{20pt} \le \sum_{m_1,m_2=0}^{2^p} \bigl|E(X_{pm_1}X_{pm_2})E(Y_{qn(m_1)}Y_{qn(m_2)}) \langle 2^{-i}\chi_{ij}, \varphi_{pm_1}\chi_{qn(m_1)}\rangle \langle 2^{-i}\chi_{ij}, \varphi_{pm_2}\chi_{qn(m_2)}\rangle\bigr|.
   \end{align*}
   Let us write $M_j := \{m: 0 \le m \le 2^p, \langle \chi_{ij}, \varphi_{pm}\chi_{qn(m)}\rangle\neq 0\}$. We also write $\rho'$ for the conjugate exponent of $\rho$, i.e. $1/\rho + 1/\rho' = 1$. H\"older's inequality and Lemma~\ref{l:schauder coefficients of iterated integrals} imply
   \begin{align*}
      &\sum_{m_1,m_2 \in M_j} \bigl|E(X_{pm_1}X_{pm_2})E(Y_{qn(m_1)}Y_{qn(m_2)}) \langle 2^{-i}\chi_{ij}, \varphi_{pm_1}\chi_{qn(m_1)}\rangle \langle 2^{-i}\chi_{ij}, \varphi_{pm_2}\chi_{qn(m_2)}\rangle\bigr| \\
      &\hspace{20pt} \lesssim \Biggl(\sum_{m_1,m_2 \in M_j} \bigl|E(X_{pm_1}X_{pm_2})\bigr|^\rho\Biggr)^{1/\rho} \Biggl(\sum_{m_1,m_2 \in M_j}\bigl|E(Y_{qn(m_1)}Y_{qn(m_2)})\bigr|^{\rho'} \Biggr)^{1/\rho'} (2^{-2 (p \vee i) + p + q})^2.
   \end{align*}
   Now write $N_j$ for the set of $n$ for which $\chi_{ij} \chi_{qn}$ is not identically zero. For every $\bar{n} \in N_j$ there are $2^{p-q}$ numbers $m \in M_j$ with $n(m) = \bar{n}$. Hence
   \begin{align*}
      &\Bigl(\sum_{m_1,m_2 \in M_j}\bigl|E(Y_{qn(m_1)}Y_{qn(m_2)})\bigr|^{\rho'} \Bigr)^{1/\rho'} \\
      &\hspace{60pt} \lesssim (2^{2(p-q)})^{1/\rho'} \bigg(\Big(\max_{n_1, n_2 \in N_j} \bigl|E(Y_{qn_1}Y_{qn_2})\bigr|\Big)^{\rho'-\rho} \sum_{n_1,n_2 \in N_j}\bigl|E(Y_{qn_1}Y_{qn_2})\bigr|^{\rho} \bigg)^{1/\rho'},
   \end{align*}
   where we used that $\rho \in [1,2]$ and therefore $\rho' - \rho \ge 0$ (for $\rho'=\infty$ we interpret the right hand side as $\max_{n_1, n_2 \in N_j} |E(Y_{qn_1}Y_{qn_2})|$). Lemma~\ref{l:rho-var to dyadic generation} implies that $\bigl(\bigl|E(Y_{qn_1}Y_{qn_2})\bigr|^{\rho'-\rho}\bigr)^{1/\rho'} \lesssim 2^{-q(1/\rho - 1/\rho')}$. Similarly we apply Lemma~\ref{l:rho-var to dyadic generation} to the sum over $n_1, n_2$, and we obtain
   \begin{align*}
      & (2^{2(p-q)})^{1/\rho'} \biggl(\Big(\max_{n_1, n_2 \in N_j} \bigl|E(Y_{qn_1}Y_{qn_2})\bigr|\Big)^{\rho'-\rho} \sum_{n_1,n_2 \in N_j}\bigl|E(Y_{qn_1}Y_{qn_2})\bigr|^{\rho} \biggr)^{1/\rho'}\\
      &\hspace{60pt} \lesssim  (2^{2(p-q)})^{1/\rho'} 2^{-q(1/\rho - 1/\rho')} (|N_j| 2^{-q})^{1/\rho'} = 2^{(q \vee i)/ \rho'} 2^{-i/\rho'} 2^{2p/\rho'} 2^{q(-2/\rho'-1/\rho)} \\
      &\hspace{60pt} = 2^{(q \vee i)(1-1/\rho)} 2^{i(1/\rho-1)} 2^{2p(1-1/\rho)} 2^{q(1/\rho-2)},
  \end{align*}
   where we used that $|N_j| = 2^{(q \vee i) - i}$. Since $|M_j| = 2^{(p \vee i) - i}$, another application of Lemma~\ref{l:rho-var to dyadic generation} yields
   \begin{align*}
      \Bigl(\sum_{m_1,m_2 \in M_j} \bigl|E(X_{pm_1}X_{pm_2})\bigr|^\rho\Bigr)^{1/\rho} \lesssim 2^{(p \vee i) / \rho} 2^{-i / \rho} 2^{-p/\rho}.
   \end{align*}
   The result now follows by combining these estimates:
   \begin{align*}
      &E\Bigl[\Bigl|\sum_{m\le 2^p} \sum_{n\le 2^q} X_{pm} Y_{qn} \langle 2^{-i}\chi_{ij}, \varphi_{pm} \chi_{qn}\rangle\Bigr|^2\Bigr]\\
      &\hspace{25pt} \lesssim \Bigl(\sum_{m_1,m_2 \in M_j} \bigl|E(X_{pm_1}X_{pm_2})\bigr|^\rho\Bigr)^{1/\rho} \Bigl(\sum_{m_1,m_2 \in M_j}\bigl|E(Y_{qn(m_1)}Y_{qn(m_2)})\bigr|^{\rho'} \Bigr)^{1/\rho'} (2^{-2 (p \vee i) + p + q})^2\\
      &\hspace{25pt} \lesssim \big(2^{(p \vee i) / \rho} 2^{-i / \rho} 2^{-p/\rho}\big) \big(2^{(q \vee i)(1-1/\rho)} 2^{i(1/\rho-1)} 2^{2p(1-1/\rho)} 2^{q(1/\rho-2)}\big)\big(2^{-4 (p \vee i) + 2p + 2q} \big)\\
      &\hspace{25pt} = 2^{(p \vee i)(1/\rho - 4)} 2^{(q \vee i)(1-1/\rho)} 2^{-i} 2^{p(4-3/\rho)} 2^{q/\rho}.
   \end{align*}
\end{proof}

\begin{thm}\label{t:existence of levy area}
   Let $X\colon [0,1] \to \R^d$ be a continuous, centered stochastic process with independent components, and assume that $X$ satisfies (HC) and ($\rho$--var) for some $\rho \in [1,2)$. Then for every $\alpha \in (0,1/\rho)$ almost surely
   \begin{align*}
      \sum_{N \ge 0} \left\lVert L(S_N X, S_N X) - L(S_{N-1} X, S_{N-1} X) \right\rVert_\alpha < \infty,
   \end{align*}
   and therefore $L(X,X) = \lim_{N \rightarrow \infty} L(S_N X,S_N X)$ is almost surely $\alpha$--H\"older continuous.
\end{thm}

\begin{proof}
   First note that $L$ is antisymmetric, and in particular the diagonal of the matrix $L(S_N X, S_N X)$ is constantly zero. For $k, \ell \in \{1, \dots, d\}$ with $k \neq \ell$ we have
   \begin{align*}
      & \lVert L(S_N X^k, S_N X^\ell) - L(S_{N-1} X^k, S_{N-1} X^\ell)\rVert_\alpha \\
      &\hspace{20pt} = \Bigl\lVert\sum_{q<N} \sum_{m,n} (X^k_{Nm} X^\ell_{qn} - X^k_{qn}X^\ell_{Nm}) \int_0^\cdot \varphi_{N m}(s) \dd \varphi_{qn}(s)\Bigr\rVert_\alpha \\
      &\hspace{20pt} \le \sum_{q<N} \Bigl\lVert \sum_{m,n} X^k_{Nm} X^\ell_{qn} \int_0^\cdot \varphi_{N m}(s) \dd \varphi_{qn}(s)\Bigr\rVert_\alpha + \sum_{q<N} \Bigl\lVert \sum_{m,n} X^\ell_{Nm} X^k_{qn} \int_0^\cdot \varphi_{N m}(s) \dd \varphi_{qn}(s)\Bigr\rVert_\alpha
   \end{align*}
   Let us argue for the first term on the right hand side, the arguments for the second one being identical. Let $r \ge 1$. Using the hypercontractivity condition (HC), we obtain
   \begin{align*}
      &\sum_{i,N} \sum_{j\le2^i} \sum_{q<N} P\Bigl( \Bigl|\sum_{m,n} X^\ell_{Nm} X^k_{qn} \langle 2^{-i} \chi_{ij}, \varphi_{N m} \chi_{qn}\rangle\Bigr| > 2^{-i\alpha} 2^{-N/(2r)} 2^{-q/(2r)} \Bigr) \\
      &\hspace{100pt} \le \sum_{i,N} \sum_{j\le 2^i} \sum_{q<N} E\Bigl( \Bigl|\sum_{m,n} X^\ell_{Nm} X^k_{qn} \langle 2^{-i} \chi_{ij}, \varphi_{N m} \chi_{qn}\rangle\Bigr|^{2r}\Bigr) 2^{ i\alpha 2 r} 2^{N + q}\\
      &\hspace{100pt} \lesssim \sum_{i,N} \sum_{j\le2^i} \sum_{q<N} E\Bigl( \Bigl|\sum_{m,n} X^\ell_{Nm} X^k_{qn} \langle 2^{-i} \chi_{ij}, \varphi_{N m} \chi_{qn}\rangle\Bigr|^{2}\Bigr)^r 2^{ i\alpha 2 r} 2^{N + q}.
   \end{align*}
   Now we can apply Lemma~\ref{l:generation moment estimate} to bound this expression by
   \begin{align*}
      &\sum_{i,N} \sum_{j\le2^i} \sum_{q<N} \bigl( 2^{(N \vee i)(1/\rho - 4)} 2^{(q \vee i)(1-1/\rho)} 2^{-i} 2^{N(4-3/\rho)} 2^{q/\rho}\bigr)^r 2^{ i\alpha 2 r} 2^{N + q}\\
      &\hspace{60pt} \lesssim \sum_{i} 2^i \sum_{N\le i} \sum_{q<N} 2^{ir(2\alpha - 4)} 2^{Nr(4-3/\rho + 1/r)} 2^{qr(1/\rho+1/r)} \\
      &\hspace{80pt} + \sum_{i} 2^i \sum_{N > i} \sum_{q\le i} 2^{ir(2\alpha - 1/\rho)} 2^{Nr(1/r - 2/\rho)} 2^{qr(1/\rho+1/r)} \\
      &\hspace{80pt} + \sum_{i} 2^i \sum_{N > i} \sum_{i<q<N} 2^{ir(2\alpha - 1)} 2^{Nr(1/r - 2/\rho)} 2^{qr(1+1/r)} \\
      &\hspace{60pt} \lesssim \sum_{i} 2^{ir(2\alpha + 3/r - 2/\rho)} + \sum_{i} \sum_{N > i}  2^{ir(2\alpha + 2/r)} 2^{Nr(1/r - 2/\rho)}\\
      &\hspace{80pt} +  \sum_{i} \sum_{N > i}  2^{ir(2\alpha +1/r - 1)} 2^{Nr(1 + 2/r - 2/\rho)}.
   \end{align*}
   For $r \ge 1$ we have $1/r - 2/\rho < 0$, because $\rho < 2$. Therefore, the sum over $N$ in the second term on the right hand side converges. If now we choose $r > 1$ large enough so that $1 + 3/r - 2/\rho < 0$ (and then also $2\alpha + 3/r - 2/\rho < 0$), then all three series on the right hand side are finite. Hence, Borel-Cantelli implies the existence of $C(\omega) > 0$, such that for almost all $\omega \in \Omega$ and for all $N, i, j$ and $q<N$
   \begin{align*}
      \Bigl|\sum_{m,n} X^\ell_{Nm}(\omega) X^k_{qn}(\omega) \langle 2^{-i} \chi_{ij}, \varphi_{N m} \chi_{qn}\rangle\Bigr| \le C(\omega) 2^{-i\alpha} 2^{-N/(2r)} 2^{-q/(2r)}.
   \end{align*}
   From here it is straightforward to see that for these $\omega$ we have
   \begin{align*}
      \sum_{N=0}^\infty \left\lVert L(S_N X(\omega), S_N X(\omega)) - L(S_{N-1} X(\omega), S_{N-1} X(\omega)) \right\rVert_\alpha < \infty.
   \end{align*}
\end{proof}

\subsection{Continuous martingales}\label{ss:pathwise ito area for martingales}

Here we assume that $(X_t)_{t \in [0,1]}$ is a $d$--dimensional continuous martingale. Of course in that case it is no problem to construct the It\^{o} integral $\ito(X,\dd X)$. But to apply the results of Section~\ref{s:pathwise ito}, we still need the pathwise convergence of $\ito_k(X,\dd X)$ to $\ito(X,\dd X)$ and the uniform H\"older continuity of $\ito_k(X,\dd X)$ along the dyadics. %of the approximating integrals. %We are not claiming the greatest generality  and work under rather restrictive conditions. The main example that we have in mind is Brownian motion.

Recall that for a $d$--dimensional semimartingale $X=(X^1, \dots, X^d)$, the quadratic variation is defined as $[ X] = ([ X^i, X^j])_{1 \le i,j \le d}$. We also write $X_s X_{s,t} := (X^i_s X^j_{s,t})_{1 \le i, j \le d}$ for $s,t \in [0,1]$.

\begin{thm}\label{t:continuous martingale iterated integrals}
   Let $X=(X^1,\dots,X^d)$ be a $d$--dimensional continuous martingale. % indexed by $[0,1]$.
   Assume that there exist $p \ge 2$ and $\beta > 0$, such that $p \beta > 7/2$, and such that
   \begin{align}\label{e:martingale area assumption}
      E(|[ X ]_{s,t}|^p) \lesssim |t-s|^{2p\beta}
   \end{align}
   for all $s,t \in [0,1]$. Then $\ito_k(X,\dd X)$ almost surely converges uniformly to $\ito(X,\dd X)$. Furthermore, for all $\alpha \in (0, \beta - 1/p)$ we have $X \in \C^\alpha$ and almost surely
   \begin{align}\label{e:uniform hoelder along dyadics for martingale}
      \sup_k \sup_{0 \le \ell < \ell' \le 2^k} \frac{|\ito_k(X,\dd X)_{\ell 2^{-k}, \ell' 2^{-k}} - X_{\ell2^{-k}} X_{\ell2^{-k}, \ell' 2^{-k}}|}{|(\ell'-\ell)2^{-k}|^{2\alpha}} < \infty.
   \end{align}
   %In particular, $X$ almost surely satisfies all the conditions of Corollary~\ref{c:ito general result}.
\end{thm}

\begin{proof}
   The H\"older continuity of $X$ follows from Kolmogorov's continuity criterion. Indeed, applying the Burkholder-Davis-Gundy inequality and \eqref{e:martingale area assumption} we have
   \begin{align*}
      E(|X_{s,t}|^{2p}) \lesssim \sum_{i=1}^d E(|X^i_{s,t}|^{2p}) \lesssim \sum_{i=1}^d E(|[ X^i]_{s,t}|^p) \lesssim E(|[ X ]_{s,t}|^p) \lesssim |t-s|^{2p\beta},
   \end{align*}
   so that $X \in \C^{\alpha}$ for all $\alpha \in (0, \beta - 1/(2p))$ and in particular for all $\alpha \in (0, \beta - 1/p)$. Since we will need it below, let us also study the regularity of the It\^{o} integral $\ito(X,\dd X)$: A similar application of Burkholder-Davis-Gundy gives
   \begin{align*}
      E(|\ito(X,\dd X)_{s,t} - X_s X_{s,t}|^p)\lesssim E\Bigl( \Bigl| \int_s^t |X_r - X_s|^2 \dd|[ X ]|_s \Bigr|^{\frac{p}{2}} \Bigr).
   \end{align*}
   We apply H\"older's inequality (here we need $p \ge 2$) to obtain
   \begin{equation*}
      E\Bigl( \Bigl| \int_s^t |X_r - X_s|^2 \dd|[ X ]|_s \Bigr|^{\frac{p}{2}} \Bigr) \lesssim E\Bigl( |[ X ]|_{s,t}^{\frac{p}{2} - 1} \int_s^t  |X_r - X_s|^p \dd|[ X ]|_s \Bigr).
   \end{equation*}
   Now the inequalities by Cauchy-Schwarz and then by Burkholder-Davis-Gundy yield
   \begin{align*}
      E\Bigl( |[ X ]|_{s,t}^{\frac{p}{2} - 1} \int_s^t  |X_r - X_s|^p \dd|[ X ]|_s \Bigr) & \lesssim E\Bigl( |[ X ]|_{s,t}^{\frac{p}{2}} \sup_{r \in [s,t]} |X_r - X_s|^p  \Bigr) \\
      &\le \sqrt{E\Bigl(\sup_{r \in [s,t]}|X_r - X_s|^{2p}\Bigr)} \sqrt{E(|[ X ]|_{s,t}^p)}\\
      &\lesssim E(|[ X ]_{s,t}|^p) \lesssim |t-s|^{2p\beta}.
   \end{align*}
   The Kolmogorov criterion for rough paths,  Theorem 3.1 of~\cite{Friz2013}, now implies that
   \begin{align}\label{e:continuous martingale pr1}
      |\ito(X,\dd X)_{s,t} - X_s X_{s,t}| \lesssim |t-s|^{2 \alpha}
   \end{align}
   almost surely for all $\alpha \in (0, \beta - 1/p)$.

   Let us get to the convergence of $\ito_k(X,\dd X)$. %We need to show that $\ito_k(X,\dd X)$ almost surely converges uniformly to $\ito(X,\dd X)$, and that the uniform H\"older condition \eqref{e:uniform hoelder along dyadics for martingale} holds.
   As before, we have
   \begin{align*}
      & E(|\ito(X,\dd X)_{\ell2^{-k}, \ell'2^{-k}} - \ito_k(X,\dd X)_{\ell2^{-k}, \ell'2^{-k}}|^p) \\
      &\hspace{60pt} = E\Bigl( \Bigl| \int_{\ell 2^{-k}}^{\ell' 2^{-k}} \sum_{m = \ell}^{\ell'-1} \1_{[m2^{-k}, (m+1)2^{-k})}(r) X_{m 2^{-k},r} \dd X_s \Bigr|^p \Bigr) \\
      &\hspace{60pt} \lesssim E\Bigl( |[ X]|_{\ell 2^{-k}, \ell' 2^{-k}}^{\frac{p}{2}-1} \int_{\ell 2^{-k}}^{\ell' 2^{-k}} \Bigl|\sum_{m = \ell}^{\ell'-1} \1_{[m2^{-k}, (m+1)2^{-k})}(r) |X_{m 2^{-k},r}|^2 \Bigr|^{\frac{p}{2}} \dd|[ X ]|_s \Bigr).
   \end{align*}
   Since the terms in the sum all have disjoint support, we can pull the exponent $p/2$ into the sum, from where we conclude that
   \begin{align*}
      &E\Bigl( |[ X]|_{\ell 2^{-k}, \ell' 2^{-k}}^{\frac{p}{2}-1} \int_{\ell 2^{-k}}^{\ell' 2^{-k}} \sum_{m = \ell}^{\ell'-1} \1_{[m2^{-k}, (m+1)2^{-k})}(r) |X_{m 2^{-k},r}|^p  \dd|[ X ]|_s \Bigr)\\
      &\hspace{70pt} \lesssim \sqrt{E\Bigl(\sup_{r \in [s,t]} \Bigl| \sum_{m = \ell}^{\ell'-1} \1_{[m2^{-k}, (m+1)2^{-k})}(r) |X_{m 2^{-k},r}|^p \Bigr|^2 \Bigr)} \sqrt{E( |[ X]|_{\ell 2^{-k}, \ell' 2^{-k}}^p)}\\
      &\hspace{70pt} \lesssim \sqrt{\sum_{m=\ell}^{\ell'-1} E( |[ X]_{m 2^{-k},(m+1)2^{-k}}|^p)} \sqrt{E( |[ X]_{\ell 2^{-k}, \ell' 2^{-k}}|^p)}\\
      &\hspace{70pt} \lesssim \sqrt{(\ell'-\ell) (2^{-k})^{2p\beta}} \sqrt{|(\ell'-\ell) 2^{-k}|^{2p\beta}} = (\ell'-\ell)^{\frac{1}{2} + p\beta} 2^{-k 2 p \beta}.
   \end{align*}
   Hence, we obtain for $\alpha \in \R$ that
   \begin{align*}
      &P\left(|\ito(X,\dd X)_{\ell2^{-k}, \ell'2^{-k}} - \ito_k(X,\dd X)_{\ell2^{-k}, \ell'2^{-k}}| > |(\ell'-\ell)2^{-k}|^{2\alpha}\right) \\
      &\hspace{160pt} \lesssim \frac{(\ell'-\ell)^{\frac{1}{2} + p\beta} 2^{-k 2 p \beta}}{(\ell'-\ell)^{2p\alpha} 2^{-k 2 p \alpha}} = (\ell'-\ell)^{\frac{1}{2} + p\beta - 2p\alpha} 2^{-k2p (\beta - \alpha)}.
   \end{align*}
   If we set $\alpha = \beta - 1/(2p) - \varepsilon$, then $1/2 + p \beta - 2p\alpha = 3/2 - p \beta + 2p\varepsilon$. Now by assumption $p \beta > 7/2$ and therefore we can find $\alpha \in (0, \beta - 1/(2p))$ such that
   \begin{align}\label{e:continuous martingale pr2}
      1/2 + p \beta - 2p\alpha < -2.
   \end{align}
   Estimating the double sum by a double integral, we easily see that for all $\gamma<-2$
   \begin{align*}
      \sum_{\ell=1}^{2^k} \sum_{\ell'=\ell+1}^{2^k} (\ell'-\ell)^{\gamma} \lesssim 2^k.
   \end{align*}
   Therefore, we have for $\alpha \in (0, \beta - 1/(2p))$ satisfying \eqref{e:continuous martingale pr2}
   \begin{align*}
      &\sum_{\ell=1}^{2^k} \sum_{\ell'=\ell+1}^{2^k} P\left(|\ito(X,\dd X)_{\ell2^{-k}, \ell'2^{-k}} - \ito_k(X,\dd X)_{\ell2^{-k}, \ell'2^{-k}}| > |(\ell'-\ell)2^{-k}|^{2\alpha}\right) \\
      &\hspace{50pt} \lesssim 2^k 2^{-k2p (\beta - \alpha)}.
   \end{align*}
   Since $\alpha < \beta - 1/(2p)$, this is summable in $k$, and therefore Borel-Cantelli implies that
   \begin{align}\label{e:continuous martingale pr3}
      \sup_k \sup_{0 \le \ell < \ell' \le 2^k} \frac{|\ito(X,\dd X)_{\ell 2^{-k}, \ell' 2^{-k}} - \ito_k(X,\dd X)_{\ell 2^{-k}, \ell' 2^{-k}} |}{|(\ell'-\ell)2^{-k}|^{2\alpha}} < \infty
   \end{align}
   almost surely. We only proved this for $\alpha$ close enough to $\beta - 1/(2p)$, but of course then it also holds for all $\alpha'\le\alpha$. %, since $(\ell'-\ell)2^{-k} \le 1$.
   The estimate \eqref{e:uniform hoelder along dyadics for martingale} now follows by combining \eqref{e:continuous martingale pr1} and \eqref{e:continuous martingale pr3}. The uniform convergence of $\ito_k(X,\dd X)$ to $\ito(X,\dd X)$ follows from \eqref{e:continuous martingale pr3} in combination with the H\"older continuity of $X$.
\end{proof}

\begin{ex}
   The conditions of Theorem~\ref{t:continuous martingale iterated integrals} are satisfied by %the $d$--dimensional standard Brownian motion. Here we can take $\beta = 1/2$, and $p$ can be taken arbitrarily large. More generally, the conditions are satisfied by
   all It\^{o} martingales of the form $X_t = X_0 + \int_0^t \sigma_s \dd W_s$, as long as $\sigma$ satisfies $E(\sup_{s \in [0,1]} |\sigma_s|^{2p}) < \infty$
%   \begin{align*}
%      E\left(\sup_{s \in [0,1]} |\sigma_s|^{2p}\right) < \infty
%   \end{align*}
   for some $p > 7$. In that case we can take $\beta = 1/2$ so that in particular $\beta - 1/p > 1/3$, which means that $X$ and $\ito(X,\dd X)$ are sufficiently regular to apply the results of Section~\ref{s:pathwise ito}.
\end{ex}

\section{Pathwise stochastic differential equations}\label{s:sde}

We are now ready to solve SDEs of the form
\begin{equation}\label{e:sde}
    \dd y(t) = b(y(t)) \dd t + \sigma(y(t)) \dd v(t), \qquad y(0) = y_0,
\end{equation}
pathwise, where the ``stochastic'' integral $\dd v$ will be interpreted as $I(\sigma(y), \dd v)$ or $\ito(\sigma(y), \dd v)$.

Assume for example that $(v, L(v,v)) \in \C^\alpha \times \C^{2\alpha}$ for some $\alpha \in (1/3,1/2)$ are given, and that $b$ is Lipschitz continuous whereas $\sigma \in C^{1+\varepsilon}_b$ for some $\varepsilon$ with $2(\alpha+\varepsilon)>1$. Then Corollary~\ref{c:controlled under smooth} implies that $\sigma(y) \in \D^{\varepsilon\alpha}_v$ for every $y \in \D^\alpha_v$, and Theorem~\ref{t:rough path integral} then shows that $y_0 + \int_0^\cdot b(y(t)) \dd t + I(\sigma(y),\dd v) \in \D^\alpha_v$. Moreover, if we restrict ourselves  to the set
\[
   \mathcal{M}_\sigma = \{ y \in \D^\alpha_v : \lVert y^v \rVert_\infty \le \lVert \sigma \rVert_\infty \},
\]
then the map $\mathcal{M}_\sigma \ni (y,y^v) \mapsto \Gamma(y) =  (y_0 + \int_0^\cdot b(y(t)) \dd t + I(\sigma(y),\dd v), \sigma(y)) \in \mathcal{M}_\sigma$ satisfies the bound
\begin{align*}
   \lVert \Gamma(y) \rVert_{v,\alpha} & \lesssim |y_0| + |b(0)| + \lVert b \rVert_{\mathrm{Lip}} \lVert y \rVert_\infty + \lVert \sigma(y) \rVert_{v,\varepsilon\alpha} (\lVert v \rVert_\alpha + \lVert v \rVert_\alpha^2 + \lVert L(v,v)\rVert_{2\alpha}) + \lVert \sigma(y) \rVert_\alpha \\
   & \lesssim |y_0| + |b(0)| + (1 + \lVert b \rVert_{\mathrm{Lip}})(1 + \lVert \sigma \rVert_{C^{1+\varepsilon}_b}^{2+\varepsilon})(1 + \lVert v \rVert_\alpha^2 + \lVert L(v,v)\rVert_{2\alpha})(1 + \lVert y \rVert_{v,\varepsilon\alpha}),
\end{align*}
where we wrote $\lVert b \rVert_{\mathrm{Lip}}$ for the Lipschitz norm of $b$.

To pick up a small factor we apply a scaling argument. For $\lambda\in(0,1]$ we introduce the map $\Lambda_\lambda \colon \C^\beta \to \C^\beta$ defined by $\Lambda_\lambda f(t) = f(\lambda t)$. Then for $\lambda = 2^{-k}$ and on the interval $[0,\lambda]$ equation~\eqref{e:sde} is equivalent to
\begin{equation}\label{e:sde rescaled}
   \dd y^\lambda(t) = \lambda b(y^\lambda(t)) \dd t + \lambda^\alpha \sigma(y^\lambda(t)) \dd v^\lambda(t), \qquad y^\lambda(0) = y_0,
\end{equation}
where $y^\lambda = \Lambda_\lambda y$, $v^\lambda = \lambda^{-\alpha} \Lambda_\lambda v$. To see this, note that
\[
   \Lambda_\lambda I(f,\dd v) = \lim_{N\to \infty} \int_0^{\lambda\cdot} S_N f (t) \partial_t S_N v (t) \dd t = \lim_{N\to\infty} \int_0^{\cdot} (\Lambda_\lambda S_N f) (t) \partial_t (\Lambda_\lambda S_N v) (t) \dd t.
\]
But now $\Lambda_{2^{-k}} S_N g = S_{N-k} \Lambda_\lambda g$ for all sufficiently large $N$, and therefore
\[
   \Lambda_\lambda I(f,\dd v) = \lambda^\alpha I(\Lambda_\lambda f, \dd v^\lambda).
\]
For the quadratic covariation we have
\[
   \Lambda_\lambda [f,v] = [\Lambda_\lambda f, \Lambda_\lambda v] = \lambda^\alpha [\Lambda_\lambda f, v^\lambda],
\]
from where we get~\eqref{e:sde rescaled} also in the It\^o case. In other words we can replace $b$ by $\lambda b$, $\sigma$ by $\lambda^\alpha \sigma$, and $v$ by $v^\lambda$.

It now suffices to show that $v^\lambda$, $L(v^\lambda, v^\lambda)$, and $[v^\lambda, v^\lambda]$ are uniformly bounded in $\lambda$. Since only increments of $v$ appear in~\eqref{e:sde} we may suppose $v(0) = 0$, in which case it is easy to see that $\lVert \Lambda_\lambda v \rVert_\alpha \lesssim \lambda^\alpha \lVert v \rVert_\alpha$ and $\lVert [v^\lambda, v^\lambda]\rVert_{2\alpha} \lesssim \lVert [v,v]\rVert_{2\alpha}$. As for the L\'evy area, we have
\begin{align*}
   L(v^\lambda, v^\lambda) & = I(v^\lambda, \dd v^\lambda) - \pi_<(v^\lambda, v^\lambda) - S(v^\lambda, v^\lambda) = \lambda^{-2\alpha} \Lambda_\lambda I(v,\dd v) - \pi_<(v^\lambda, v^\lambda) - S(v^\lambda, v^\lambda) \\
   & = \lambda^{-2\alpha} \big\{ \Lambda_\lambda L(v,v) + [\Lambda_\lambda \pi_<(v,v) - \pi_<(\Lambda_\lambda v,\Lambda_\lambda v)] + [\Lambda_\lambda S(v,v) - S(\Lambda_\lambda v, \Lambda_\lambda v)]\big\},
\end{align*}
and therefore
\[
   \lVert L(v^\lambda, v^\lambda) \rVert_{2\alpha} \lesssim \lVert L(v,v) \rVert_{2\alpha} + \lVert S(v,v) \rVert_{2\alpha} + \lVert v \rVert_{\alpha}^2 + \lambda^{-2\alpha} \lVert \Lambda_\lambda \pi_<(v,v) - \pi_<(v^\lambda,v^\lambda) \rVert_{2\alpha}.
\]
But now
\begin{align*}
   |\Lambda_\lambda \pi_<(v,v)_{s,t} - \pi_<(v^\lambda,v^\lambda)_{s,t}| & \le |  \pi_<(v,v)_{\lambda s,\lambda t} - v(\lambda s) v_{\lambda s, \lambda t}| \\
   &\quad + |\Lambda_\lambda v(s) (\Lambda_\lambda v)_{s,t}  - \pi_<(\Lambda_\lambda v, \Lambda_\lambda v)_{s,t}| \\
   &\lesssim \lVert v \rVert_\alpha^2 |\lambda(t-s)|^{2\alpha} + \lVert \Lambda_\lambda v \rVert_\alpha |t-s|^{2\alpha} \\
   &\lesssim \lambda^{2\alpha} \lVert v \rVert_\alpha^2 |(t-s)|^{2\alpha}.
\end{align*}

From here we obtain the uniform boundedness of $\lVert v^\lambda \rVert_{v^\lambda,\alpha}$ for small $\lambda$, depending only on $b,\sigma, v, L(v,v)$ and possibly $[v,v]$, but not on $y_0$. If $\sigma \in C^{2+\varepsilon}_b$, similar arguments give us a contraction for small $\lambda$, and therefore we obtain the existence and uniqueness of solutions to~\eqref{e:sde rescaled}. Since all operations involved depend on $(v,L(v,v),y_0)$ and possibly $[v,v]$ in a locally Lipschitz continuous way, also $y^\lambda$ depends locally Lipschitz continuously on this extended data.

Then $y = \Lambda_{\lambda^{-1}} y^\lambda$ solves~\eqref{e:sde} on $[0,\lambda]$, and since $\lambda$ can be chosen independently of $y_0$, we obtain the global in time existence and uniqueness of a solution which depends locally Lipschitz continuously on $(v, L(v,v), y_0)$ and possibly $[v,v]$.

\begin{thm}\label{t:sde}
   Let $\alpha \in (1/3, 1)$ and let $(v,L(v,v))$ satisfy the assumptions of Theorem~\ref{t:rough path integral}. Let $y_0 \in \R^d$ and $\varepsilon>0$ be such that $\alpha(2+\varepsilon) > 2$ and let $\sigma \in C^{2+\varepsilon}_b$ and $b$ be Lipschitz continuous. Then there exists a unique $y \in \D^\alpha_v$ such that
   \[
      y = y_0 + \int_0^\cdot b(y(t)) \dd t + I(\sigma(y), \dd v).
   \]
   The solution $y$ depends locally Lipschitz continuously on $(v, L(v,v), y_0)$. If furthermore $[v,v]$ satisfies the assumptions of Corollary~\ref{c:pathwise ito with smooth quadratic variation}, then there also exists a unique solution $x\in \D^\alpha_v$ to
   \begin{align*}
      x & = y_0 + \int_0^\cdot b(x(t)) \dd t + \ito(\sigma(x), \dd v) \\
      & = y_0 + \int_0^\cdot b(x(t)) \dd t + I(\sigma(x), \dd v) - \frac{1}{2} \int_0^\cdot \dD \sigma(x(t)) \sigma(x(t)) \dd [v,v]_t
   \end{align*}
   and $x$ depends locally Lipschitz continuously on $(v, L(v,v), [v,v], y_0)$.
\end{thm}

\begin{rmk}
   Since our integral is pathwise continuous, we can of course consider anticipating initial conditions and coefficients. Such problems arise naturally in the study of random dynamical systems; see for example~\cite{Imkeller1998,Arnold1999}. There are various approaches, for example filtration enlargements, Skorokhod integrals, or the noncausal Ogawa integral. While filtration enlargements are technically difficult, Skorokhod integrals have the disadvantage that in the anticipating case the integral is not always easy to interpret and can behave pathologically; see~\cite{Barlow1995}. With classical rough path theory these technical problems disappear. But then the integral is given as limit of compensated Riemann sums (see Proposition~\ref{p:Gubinelli rough paths}). With our formulation of the integral it is clear that we can indeed consider usual Riemann sums. An approach to pathwise integration which allows to define anticipating integrals without many technical difficulties while retaining a natural interpretation of the integral is the stochastic calculus via regularization of Russo and Vallois~\cite{Russo1993,Russo2007}. The integral notion studied by Ogawa~\cite{Ogawa1984, Ogawa1985} for anticipating stochastic integrals with respect to Brownian motion is based on Fourier expansions of integrand and integrator, and therefore related to our and the Stratonovich integral (see Nualart, Zakai~\cite{NualartZakai1989}). Similarly as the classical It\^o integral, it is interpreted in an $L^2$ limit sense, not a pathwise one.
\end{rmk}

\appendix

\section{Regularity for Schauder expansions with affine coefficients}\label{a:schauder with affine coefficients}

Here we study the regularity of series of Schauder functions that have affine functions as coefficients. First let us establish an auxiliary result.

\begin{lem}
   Let $s < t$ and let $f:[s,t] \rightarrow \L(\R^d,\R^n)$ and $g:[s,t]\rightarrow \R^d$ be affine functions. Then for all $r \in (s,t)$ and for all $h>0$ with $r-h \in [s,t]$ and $r+h \in [s,t]$ we have
   \begin{align}\label{e:second order increments quadratic}
      |(fg)_{r-h,r} - (fg)_{r,r+h}| \le 4 |t-s|^{-2} h^2 \lVert f\rVert_\infty \lVert g \rVert_\infty.
   \end{align}
\end{lem}

\begin{proof}
   For $f(r) = a_1 + (r-s)b_1$ and $g(r) = a_2 + (r-s)b_2$ we have
   \[
      |(fg)_{r-h,r} - (fg)_{r,r+h}| = | 2 f(r) g(r) - f(r-h) g(r-h) - f(r+h)g(r+h)| = |-h^2 b_1 b_2|.
   \]
   Now $f_{s,t} = b_1(t-s)$ so that $|b_1| \le 2 |t-s|^{-1} \lVert f \rVert_\infty$, and similarly for $b_2$.
\end{proof}

Now we are ready to prove the regularity estimate.

\begin{lem}\label{l:upm hoelder appendix}
   Let $\alpha \in (0,2)$ and let $(u_{pm})\in \A^\alpha$. Then $\sum_{p,m} u_{pm} \varphi_{pm} \in \C^\alpha$ and
   \begin{align*}
      \Bigl\lVert \sum_{p,m} u_{pm} \varphi_{pm}\Bigr\rVert_\alpha \lesssim \lVert (u_{pm}) \rVert_{\A^\alpha}.
   \end{align*}
\end{lem}

\begin{proof}
   We need to examine the coefficients $2^{-q} \langle \chi_{qn}, \dd(\sum_{pm} u_{pm} \varphi_{pm})\rangle$. The cases $(q,n) = (-1,0)$ and $(q,n)=(0,0)$ are easy, so
%   First let us treat the cases $(q,n) = (-1,0)$ and $(q,n)=(0,0)$. Here it suffices to estimate the uniform norm of $\sum_{pm} u_{pm} \varphi_{pm}$. For fixed $p$, the functions $(\varphi_{pm})_m$ have disjoint support, and therefore
%   \begin{align*}
%      \sum_p \left\lVert \sum_m u_{pm} \varphi_{pm}\right\rVert_\infty \le \sum_p \max_m \lVert u_{pm} \rVert_\infty \le \sum_p 2^{-p\alpha} \lVert (u_{pm})\rVert_{\A^\alpha} \lesssim \lVert (u_{pm})\rVert_{\A^\alpha}.
%   \end{align*}
   let $q \ge 0$ and $1 \le n \le 2^q$. If $p>q$, then $\varphi_{pm}(t^i_{qn}) = 0$ for $i=0,1,2$ and for all $m$, and therefore
   \begin{align*}
      2^{-q}\Big\langle \chi_{qn}, \dd\Big(\sum_{p,m} u_{pm} \varphi_{pm}\Big)\Big\rangle = 2^{-q}\sum_{p\le q} \sum_m \langle \chi_{qn}, \dd(u_{pm} \varphi_{pm})\rangle.
   \end{align*}
   If $p < q$, there is at most one $m_0$ with $\langle \chi_{qn}, \dd(u_{pm} \varphi_{pm})\rangle \neq 0$. The support of $\chi_{qn}$ is then contained in $[t^0_{pm_0},t^1_{pm_0}]$ or in $[t^1_{pm_0}, t^2_{pm_0}]$ and $u_{pm}$ and $\varphi_{pm}$ are affine on these intervals, so~\eqref{e:second order increments quadratic} yields
   \begin{align*}
      \sum_m |2^{-q}\langle \chi_{qn}, \dd (u_{pm} \varphi_{pm})\rangle| & = \sum_m |(u_{pm} \varphi_{pm})_{t^0_{qn},t^1_{qn}} - (u_{pm} \varphi_{pm})_{t^1_{qn},t^2_{qn}}| \\
      & \lesssim 2^{2p} 2^{-2q} \lVert u_{pm}\rVert_\infty \lVert \varphi_{pm} \rVert_\infty \lesssim 2^{p(2-\alpha)-2q} \lVert (u_{pm})\rVert_{\A^\alpha}.
   \end{align*}
   For $p=q$ we have $\varphi_{qn}(t^0_{qn}) = \varphi_{qn}(t^2_{qn}) = 0$ and $\varphi_{qn}(t^1_{qn}) = 1/2$, and thus
   \[
      \sum_m |2^{-q}\langle \chi_{qn}, \dd(u_{qm} \varphi_{qm})\rangle| = \left|(u_{qn} \varphi_{qn})_{t^0_{qn},t^1_{qn}}-(u_{qn} \varphi_{qn})_{t^1_{qn},t^2_{qn}}\right| = |u(t^1_{qn})|\lesssim 2^{-\alpha q} \lVert (u_{pm})\rVert_{\A^\alpha}. %= 2^{p(2-\alpha) -2q} \lVert (u_{pm})\rVert_{\A^\alpha}.
   \]
   Combining these estimate and using that $\alpha < 2$, we obtain
   \[
      2^{-q}\Big|\Big\langle \chi_{qn}, \dd\Big(\sum_{pm} u_{pm} \varphi_{pm}\Big)\Big\rangle\Big| \lesssim \sum_{p\le q} 2^{p(2-\alpha)-2q} \lVert (u_{pm})\rVert_{\A^\alpha} \simeq 2^{-\alpha q} \lVert (u_{pm})\rVert_{\A^\alpha},
   \]
   which completes the proof.
   %where we used that $\alpha < 2$ and therefore $\sum_{p\le q}2^{p(2-\alpha)} \simeq 2^{q(2-\alpha)}$.
\end{proof}

\bibliography{quellen}{}
\bibliographystyle{amsalpha}

\end{document}